\documentclass[a4paper,12pt]{amsart}
\usepackage[english]{babel}
\usepackage[latin1]{inputenc}
\usepackage[T1]{fontenc}
\usepackage{amsthm}
\usepackage{latexsym,amsmath,amsfonts,amscd,amssymb}
\usepackage{graphicx}
\usepackage{xypic}
\usepackage{array,booktabs,arydshln}
\usepackage{tikz,enumitem,multicol}
\usetikzlibrary{decorations.text}
\usetikzlibrary{arrows}
\usepackage{pgfplots}
\usetikzlibrary{arrows.meta}
\usetikzlibrary{intersections}
\usepgfplotslibrary{fillbetween}
\usepackage{tkz-tab}
\usepackage{tabularx}
\usepackage{setspace}
\usepackage{fancyhdr}
\usepackage{fancybox}
\usepackage[all,cmtip]{xy}
\usepackage{pdfpages}
\usepackage{enumitem}
\usepackage{calc}
\usepackage{chngcntr}
\usepackage[pagewise]{lineno}

\usepackage{newtxtext}
\usepackage{xpatch}

\patchcmd{\part}{\normalfont}{\alessandrofont}{}{}
\patchcmd{\specialsection}{\normalfont}{\alessandrofont}{}{}
\patchcmd{\section}{\normalfont\scshape}{\alessandrofont}{}{}
\patchcmd{\subsection}{\normalfont}{\alessandrofont}{}{}
\patchcmd{\subsubsection}{\normalfont}{\alessandrofont}{}{}
\newcommand{\alessandrofont}{\normalfont\sffamily}

\newlength\mylength
\newtheorem{theorem}{Theorem}[section]
\newtheorem{mthm}{Main Theorem}
\newtheorem{proposition}[theorem]{Proposition}

\newtheorem{question}{Question}
\newtheorem{lemma}[theorem]{Lemma}
\newtheorem{corollary}[theorem]{Corollary}
\newtheorem{algorithm}[theorem]{Algorithm}
\theoremstyle{definition}
\newtheorem{rmk}[theorem]{Remark}
\newtheorem{example}[theorem]{Example}
\newtheorem{definition}[theorem]{{\bf Definition}}

\newtheorem*{definition*}{{\bf Definition}}
\counterwithin{figure}{theorem}

\newcommand{\opn}{{\mathcal{O}_{\mathbb{P}^3}}}

\newcommand{\HH}{{\mathcal H}}

\newcommand{\Bb}{\mathcal{B}^{\beta}}

\newcommand{\Pic}{\operatorname{Pic}}

\newcommand{\Hom}{\operatorname{Hom}}

\newcommand{\Ext}{\operatorname{Ext}}
\newcommand{\im}{\operatorname{Im}}

\newcommand{\ch}{\operatorname{ch}}
\newcommand{\Coh}{\operatorname*{Coh}(X)}
\newcommand{\Db}{\operatorname*{D}^b(X)}

\newcommand{\Shione}{\langle \operatorname{S}(\tilde{\mathcal{E}}) \rangle}
\newcommand{\Shial}{\langle \operatorname{S}(\mathcal{E}_1) \rangle}
\newcommand{\Shitwo}{\langle \operatorname{S}(\mathcal{E}_2) \rangle}
\newcommand{\Shi}{\langle \operatorname{S}(\mathcal{E}) \rangle}
\newcommand{\Stab}{\operatorname{Stab}}
\newcommand{\lba}{\lambda_{\beta,\alpha}}
\newcommand{\KnumX}{\operatorname{K}_{0}(X)}
\newcommand{\UpH}{\bar{\mathbb{H}}}
\newcommand{\Id}{\operatorname{Id}}
\newcommand{\Tr}{\operatorname{Tr}}

\newcommand{\Gl}{\operatorname{GL}}

\newcommand{\Dbal}{\operatorname{D^b}}

\newcommand{\Supp}{\operatorname{Supp}}

\newcommand{\Knum}{\operatorname{K}_{0}}

\newcommand{\Rep}{\operatorname{rep}}
\newcommand{\Ker}{\operatorname{Ker}}
\newcommand{\Real}{\operatorname{Re}}
\newcommand{\Imag}{\operatorname{im}}

\title{Determinants, even instantons and Bridgeland stability}
\author{Victor do Valle Pretti}
\address{Universidade de S\~ao Paulo \\
Instituto de Matem\'atica e Estat\'istica \\ Rua do Mat\~ao,1010  \\
05508-090 S\~ao Paulo, Brasil\\       
uni-pretti.w3qpo@simplelogin.com}
\begin{document}

\maketitle

\begin{abstract}
We provide a systematic way of calculating the quiver region associated to a given exceptional collection, and as an application we prove that $\mu$-stable sheaves represented by $2$-step complexes are Bridgeland stable by using the determinant. In the later sections, we focus on the case of even rank $2$ instantons over $\mathbb{P}^3$ and $Q_3$ to prove that the instanton sheaves, instanton bundles and perverse instantons are Bridgeland stable and provide a description of the moduli space near their only actual  wall.
\end{abstract}

\section*{Introduction}

In \cite{Mac}, Macr\`{i} studied topological properties of the space of Bridgeland stability conditions  using that certain Bridgeland stability conditions are equivalent to quiver stability, as in \cite{King}, via a compatible Ext-exceptional collection. Recently, his work was further improved by Chen's paper, \cite{Chen}, describing a simply connected region of the space of stability condition on $\mathbb{P}^3$.

Exceptional collections have also been used to prove the projectivity of the moduli space of semistable objects with respect to a Bridgeland stability condition by associating to it a quiver moduli space, see \cite{Bert, Arc}, and also these can be used to describe explicit examples of Bridgeland stable objects as in Mu's work \cite{Mu}. The latter was a source of inspiration to produce this paper. 

 The most natural example in this situation are the \emph{instanton sheaves}. These sheaves have been heavily studied throughout the last few decades by many authors such as Okonek--Splinder, Atiyah, Hartshorne and many others \cite{AB, Cos, Har, Jar, OS}. One way to define the instanton sheaves was proposed by Atiyah--Barth, \cite{AB}, as the cohomology of a monad \[ 0 \rightarrow \opn^{\oplus c}(-1) \overset{A}{\rightarrow} \opn^{\oplus 2c+r} \overset{B}{\rightarrow} \opn^{\oplus c}(1) \rightarrow 0,\] where the number $c$ is called the charge and $r$ the rank of the instanton sheaf $I=\Ker(B)/\operatorname{Im}(A)$. This definition was generalized independently to other Fano varieties by Faenzi and Kuznetsov, \cite{Faen, Kuz}, by considering instanton bundles as cohomology of linear complexes of exceptional objects in an exceptional collection generating the bounded derived category $\Db$, making them the perfect candidates for examples.

For $\mathbb{P}^3$, the picture is even more explicit due to the work of Jardim, Jardim--Gargate and Jardim--Maican--Tikhomirov, \cite{Jar,JRM,JMT}. In \cite{JMT} the authors prove that the singularity sheaf of an instanton sheaf, a rank $0$ instanton, can be resolved by a free resolution using the bundles $\opn(-i)$, with the integer $i$ ranging from $-1$ to $1$. This allowed us to describe exactly a chamber where the instanton sheaves are Bridgeland stable.

Recently, a number of papers delved into the question of whether the instantons over Fano threefolds are Bridgeland stable, \cite{Qin1,Qin2,ZL}. A common theme in their proofs is the use of the Kuznetsov component and the fact that instantons of minimal charge are in the Kuznetsov component. This approach is necessary due to the nonexistence of a full strong exceptional collection in some Fano threefolds. On the other hand, with our approach in this paper, we are able to achieve their stability for every charge for some of the cases of Fano threefolds where these exceptional collections are known to exist.

The paper is organized as follows: the first two sections recall two already established theories of stability, Bridgeland and quiver stability. Section \ref{Sec-ExcCo} deals with a general theory of exceptional collections and a base for the theory developed later in the paper. Most of what is done in these sections is well known. 

In Section \ref{Sec-QuivR} we define and describe how to calculate the quiver region in a systematic way, see Proposition \ref{Prop-Upper-Half-plane condition}. We finish this section by providing several examples of applications of this proposition to Fano threefolds. Section \ref{Sec-TecLin} provides the technical framework for our work by describing the determinant condition in Lemma \ref{Lemma-Determinant Condition} and proving its relation with $\mu$-stability, which in turn provides an open subset inside the moduli of stable representations of the Kronecker quiver.

In the last section, we describe the stability of instanton sheaves, instanton bundles and of a special class perverse instantons for $\mathbb{P}^3$, and the stability of even instanton bundles on $Q_3$, of any possible charge. For the case of $\mathbb{P}^3$, we also describe the whole structure of their walls inside the quiver region. Furthermore, a stratification of their moduli space by the moduli of stable representations of the Kronecker quiver is also given. 

These results are summarized in Main Theorems \ref{Main Theorem - 3} and \ref{Main Theorem - 4}, where we describe the moduli space of Bridgeland (semi)stable objects with Chern character $(-2,0,c,0)$. These moduli spaces are obtained from crossing the actual walls determined by the resulting object after applying the truncation functor to the chain complex representing the objects in the moduli space. The rank $2$ instantons, after a shift by $[1]$, form an open subset of these moduli spaces.

The Appendix is a collection of intermediate results obtained during the early stages of this project that were useful for doing calculations and early speculations on the scope of this approach.

\textbf{Acknowledgments:} I would like to thank both my advisors Marcos Jardim and Antony Maciocia for the opportunity to work on this topic and the incredible support in the development of this research. I also appreciate Daniele Faenzi for his insights in the theory of intantons on Fano varieties, and my friend Charles Almeida for the helpful discussions. Special thanks to Dapeng Mu for carefully reading the paper and pointing out an oversight on the calculations for the $V_5$ quiver region, and Daniel Bernal for his helpful comments. Finally, I am grateful to the referee for their valuable corrections and insights, which greatly enhanced this work. This work was supported by FAPESP research grant 2016/25249-0, 2019/05207-9, 2022/12883-3 and CAPES research grant 88887.647097/2021-00. This project was partially developed at the University of Edinburgh to which I am grateful for the hospitality and support.

\vspace{10pt}

\textbf{Notation:}\begin{itemize}
    \item Let $X$ be a smooth projective variety over $\mathbb{C}$.
    \item $\KnumX$ is the Grothendieck Group of the coherent sheaves over $X$.
    \item Given $E\in \Db$ and $k$ a positive integer, we say that $E$ is represented by a $k$-step complex if it is isomorphic to a complex $(A^\bullet,d)$ such that there exists $l$ with: $A^i=0$ for $i\neq l+1,...,l+k$.
    \item $\mathbb{H}$ is the strict upper-half plane $\{(\beta,\alpha)\ | \ \text{$\beta \in \mathbb{R}$ and $ \alpha> 0$ }\}$ and the upper-half plane $\UpH= \{(\beta,\alpha)\ |\  \text{$\beta \in \mathbb{R}$ and either $ \alpha> 0$ or $\alpha=0$ and $\beta<0$}\}$.
    \item $\ch_i^{\beta H}(E)$ is the object in $i$-Chow ring $A^i(X)\otimes\mathbb{Q}$ determined by \[ \ch_i^{\beta H}(E)=(e^{\beta H}\ch(E))_i= \Sigma_{j+k=i}\ch_j(E)H^k\frac{\beta^k}{k!}, \] where $H$ is an ample class in the $1$-Chow group and $E\in \Db=\Dbal(\Coh)$.
    \item For any heart of a bounded $t$-structure $\mathcal{A}$ of $\Db$, there exists a natural cohomology functor $\mathcal{H}_\mathcal{A}:\Db \rightarrow \mathcal{A}$.
    \item The only graph that appears in this text that does not represent the $(\beta,\alpha)$-plane of (weak) stability conditions is displayed in Figure \ref{Half-plane condition}, representing the complex plane $\mathbb{C}$.
\end{itemize}

\section{Bridgeland stability}\label{Sec-Bri}

We recall the well-known notions of Bridgeland stability conditions. In particular, we discuss a few Fano varieties where the double-tilt construction yields a geometric stability condition. The definition of stability condition will be done through the notion of \emph{slicing}, just as in \cite{Bri1}. 

    \begin{definition}
    A \emph{slicing} $\mathcal{P}$ in $\Db$ is a collection of additive subcategories $\mathcal{P}(\phi) \subset \Db$ for every $\phi \in \mathbb{R}$ satisfying:
    \begin{itemize}
        \item $\mathcal{P}(\phi)[1]=\mathcal{P}(\phi+1)$,
        \item $\Hom(A,B)=0$ if $A \in \mathcal{P}(\phi_1)$ and $B \in \mathcal{P}(\phi_2)$ with $\phi_1>\phi_2$,
        \item For every object $E \in \Db$ we have a decomposition
        
        \qquad \xymatrix{
        0=E_0 \ar[r]  &E_1 \ar[r] \ar[d] &E_2 \ar[r] \ar[d] &... \ar[r] &E_{n-1} \ar[r] \ar[d] &E_n=E \ar[d]\\
        &A_1 \ar@{.>}[lu] &A_2 \ar@{.>}[lu] &&A_{n-1} \ar@{.>}[lu] &A_n \ar@{.>}[lu]
        }

        such that $A_i \in \mathcal{P}(\phi_i)$ and $\phi_i>\phi_{i+1}$, for every $i$.
        
    \end{itemize}
    \end{definition}

The slicing encodes the information on semistable objects with respect to the following notion of stability condition. 
    
    \begin{definition}
    A \emph{weak stability condition} is a pair $\sigma=(Z, \mathcal{P})$ consisting of a group homomorphism $Z: \KnumX \rightarrow \mathbb{C}$, known as the \emph{stability function}, and a slicing $\mathcal{P}$ satisfying the following conditions:
    \begin{itemize}
        \item For every object $E \in \mathcal{P}(\phi)$, $Z(E)= p\cdot e^{i\phi \pi}$ for some $p \in \mathbb{R}_{\geq0}$.
        
        \item The homomorphism $Z$ can be factored via a surjective map  \[v:\KnumX \rightarrow \Lambda,\] where $\Lambda$ is a finite-dimensional $\mathbb{Z}$-lattice. 
    \end{itemize}
    \end{definition}

For a given weak stability condition $\sigma=(Z,\mathcal{P})$, $E$ is said to be \emph{$\phi_\sigma$-semistable} if it is in $\mathcal{P}(\phi)$ for some $\phi \in \mathbb{R}$. 

    \begin{definition}
    A \emph{Bridgeland stability condition} $\sigma=(Z,\mathcal{P})$ is a weak stability condition such that 
    \begin{itemize}
        \item\textbf{(Positivity)} For every nonzero $E \in \mathcal{P}(\phi)$, $Z(E)= p\cdot e^{i\phi \pi}$ for some $p \in \mathbb{R}_{>0}$.
 
         \item \textbf{(Support Property)} Let $|| \cdot ||_{\mathbb{R}}$ be a norm in $\Lambda\otimes\mathbb{R}$. Then
         \begin{center}
         $ \inf \Big\{ \frac{|Z(E)|}{||v(E)||} : \text{$E \in \mathcal{P}(\phi)$ nonzero and $\phi \in \mathbb{R}$}\Big\} >0$. 
         \end{center} 
    \end{itemize}
    \end{definition}

As proved by Bridgeland in \cite[Proposition 5.3]{Bri1}, this definition is equivalent to the definition of stability conditions using hearts of bounded $t$-structures as $\sigma=(Z,\mathcal{A})$ with $Z$ a stability function and \[ \mathcal{A}=\mathcal{P}((\psi,\psi+1]):=\langle E | E \in \mathcal{P}(\phi)\text{ with $\phi \in (\psi,\psi+1]$}\rangle,\] where $\langle S \rangle$ of a given set $S \subset \Db$ is the extension-complete full subcategory obtained from $S$. The category $\mathcal{A}$ obtained in this way is an abelian subcategory of $\Db$. 

We focus on stability conditions obtained by the methods in \cite{BMT} and use the notation as in \cite{Sch1}. For that purpose, fix $\dim(X)=3$ and an ample class $H$ in $X$, $v: \KnumX \rightarrow \Lambda=\mathbb{Z}\oplus\mathbb{Z}\oplus\mathbb{Z}\frac{1}{2}\oplus\mathbb{Z}\frac{1}{6}$ defined by \[v(E)=(\ch_0(E)H^3,\ch_1(E)H^2,\ch_2(E)H,\ch_3(E)) \] and $(\beta,\alpha)\in \mathbb{H}$ so that \[ Z_{\beta,\alpha}^t(E):= -(\ch_2^{\beta H}(E)H-\frac{\alpha^2}{2}H^3\ch_0(E)) - i(\ch_1^{\beta H}(E))\] is a weak stability function. This stability condition is defined in a heart of bounded $t$-structure $\mathcal{B}^\beta$ obtained by tilting $\Coh$ using Mumford's $\mu$-slope and the torsion pairs $(\mathcal{T}_\beta,\mathcal{F}_\beta)$, as:
\[\mathcal{T}_{\beta} := \{E \in \Coh | \text{ for every $E \twoheadrightarrow Q$, $\mu(Q)>\beta$}\},\]
\[\mathcal{F}_{\beta} := \{E \in \Coh | \text{ for every $F \hookrightarrow E$, $\mu(F)\leq \beta$}\}.\] 

The weak stability condition $(Z_{\beta,\alpha}^t,\mathcal{B}^\beta)$ is known as a \emph{tilt stability condition}. In the heart $\Bb$ we also define the tilt slope of an object $E$ as the slope of the complex number $Z_{\beta,\alpha}^t(E)$, denoted by $\nu_{\beta,\alpha}(E)$ if $Z_{\beta,\alpha}^t(E)\neq0$ and $\nu_{\beta,\alpha}(E)=+\infty$ otherwise. That is,

\begin{equation*}
    \nu_{\beta,\alpha}(E)= -\frac{\Real(Z_{\beta,\alpha}^t(E)}{\Imag(Z_{\beta,\alpha}^t(E))}
\end{equation*}

The same tilting method can be used in $\mathcal{B}^\beta$ with the torsion pair
\[\mathcal{T}_{\beta,\alpha} := \{E \in \mathcal{B}^\beta | \text{ for every $E \twoheadrightarrow Q$, $\nu_{\beta,\alpha}(Q)>0$}\},\]
\[\mathcal{F}_{\beta,\alpha} := \{E \in \mathcal{B}^\beta | \text{ for every $F \hookrightarrow E$, $\nu_{\beta,\alpha}(F)\leq 0$}\}\] 

\noindent to obtain a heart of a bounded $t$-structure $\mathcal{A}^{\beta,\alpha}$, for every $(\beta,\alpha)\in \mathbb{H}$. Let $s$ be a positive real number, the Bridgeland stability function we are interested in this paper is \begin{equation}\label{Stability-Function-Gen} Z_{\beta,\alpha,s}(E)= -\ch^{\beta H}_3(E)+\left(s+1/6 \right)\alpha^2\ch_1^{\beta H}(E)H^2 + i(\ch_2^{\beta H}(E)H-\frac{\alpha^2}{2}H^3\ch_0(E)).\end{equation} The slope of the complex number $Z_{\beta,\alpha,s}(E)$ with respect to an object $E \in \mathcal{A}^{\beta,\alpha}$ will be denoted by $\lambda_{\beta,\alpha,s}(E)$. The associated slice to the stability condition $\sigma_{\beta,\alpha,s}=(Z_{\beta,\alpha,s},\mathcal{A}^{\beta,\alpha})$ will be denoted by $\mathcal{P}_{\beta,\alpha}$.

We define $\Stab(X)$ to be the \emph{space of stability conditions}; it is the given by the set of Bridgeland stability conditions and it has a natural topology as a generalized metric space, more details in \cite{Bri1}. Moreover, for a fixed $s>0$ and a variety $X$ where the above construction yields a Bridgeland stability condition, we have the continous map $\phi_s :\mathbb{H} \rightarrow \Stab(X)$ associating a point in $\mathbb{H}$ to a stability condition $\sigma_{\beta,\alpha,s}=(Z_{\beta,\alpha,s},\mathcal{A}^{\beta,\alpha})$. In this situation, we identify $\mathbb{H}$ with its image by $\phi_s$ and represent the points in the image by either a stability condition $\sigma$ or a point $(\beta,\alpha)\in \mathbb{H}$.

To simplify the notation, as in \cite{JM}, denote the negative of the real part of $Z_{\beta,\alpha,s}(E)$, $-\Real(Z_{\beta,\alpha,s}(E))$, by $\tau_{\beta,\alpha,s}(E)$ and its imaginary part by $\rho_{\beta,\alpha}(E)$. 

\begin{example}
Let $X=\mathbb{P}^3$. The construction above was shown to satisfy the necessary conditions for it to be called a Bridgeland stability condition by \cite{Mac3} and \cite{BMT} for every $(\beta,\alpha)\in \mathbb{H}$ and $s>0$. In this case we can simplify the notation by omitting the $H$ from \eqref{Stability-Function-Gen} because the rational Chow ring of $\mathbb{P}^3$ is naturally isomorphic to $\mathbb{Q}[x]/x^4$.
\end{example}

Our main examples of varieties will be smooth Fano threefolds $X$ with $\Pic(X)=\mathbb{Z}$ and with $H^3(X,\mathbb{Z})=0$, that is, having a full strong exceptional collection in $\Db$, see \cite{Faen}. Let us fix an ample generator for $\Pic(X)$ as $\mathcal{O}_{X}(H)$ such that $i_XH=-K_X$ with $i_X$ being the index of the variety $X$ and $K_X$ its canonical bundle. In this situation we will also use the notation $\mathcal{O}(k)$ for $\mathcal{O}_{X}(kH)$. These varieties are described by the following list as in \cite{Faen}:

\begin{itemize}
    \item $X=\mathbb{P}^3$ with $H^3=1$, if $i_X=4$;
    \item $X=Q_3$ with $H^3=2$, if $i_X=3$;
    \item $X=V_5$ is an hyperplane section of the Grassmannian $Gr(3,5)$ with $H^3=5$, if $i_X=2$;
    \item $X=V_{22}$ is a prime Fano threefold in $\mathbb{P}^{13}$ with $H^3=22$, if $i_X=1$.
\end{itemize}

The existence of Bridgeland stability conditions for these varieties was proved in \cite{Li1}. If one were to consider the isomorphism $\KnumX\otimes\mathbb{Q}=A(X)\otimes\mathbb{Q}=\mathbb{Q}[1,H,L,P]$ of the rational Grothendieck group and the rational Chow group, then it is possible to see that the classes of a point $P$, line $L$, hyperplane $H$ and the trivial class $1$ generate the respective group satisfying the relations $H^2=dL$ and $H\cdot L= P$.

The following definition is the standard definition of a wall on the space of weak and Bridgeland stability conditions.

\begin{definition}Let $E \in \KnumX$.
A \emph{numerical wall} inside the space of (weak)Bridgeland stability conditions for $E$ with respect to an element $F \in \KnumX$ is the subset $\Upsilon_{E,F}$ of stability conditions $\sigma=(Z,\mathcal{A})$ such that 
\[
f_\sigma(E,F) := \Real(Z(E))\Imag(Z(F)) - \Real(Z(F))\Imag(Z(E)) =0
\]

A subset of a numerical wall for $E$ with respect to $F\in \KnumX$ is called an \emph{actual wall} if, for each point $\sigma=(Z,\mathcal{A})$ in this subset, there exists an exact sequence of $\phi_\sigma$-semistable objects \[0 \rightarrow F \rightarrow E \rightarrow Q \rightarrow 0\] in $\mathcal{A}$ with $f_\sigma(E,F)=f_\sigma(E,Q)=0$.  
\end{definition}

We will be concerned with two types of numerical walls: the ones defined for tilt stability conditions $(Z_{\beta,\alpha}^t,\mathcal{B}^\beta)$, called the tilt walls or $\nu$-walls, and those defined by the Bridgeland stability condition $\sigma_{\beta,\alpha,s}=(Z_{\beta,\alpha,s},\mathcal{A}^{\beta,\alpha})$, called the $\lambda$-walls. For example, let $s \in \mathbb{R}_{>0}$ then a point $(\beta,\alpha)\in\mathbb{H}$ is at the numerical $\lambda$-wall for $E \in \KnumX$ with respect to $F$ if \[ f_{\sigma_{\beta,\alpha,s}}(E,F):=\tau_{\beta,\alpha,s}(E)\rho_{\beta,\alpha}(F)-\tau_{\beta,\alpha,s}(F)\rho_{\beta,\alpha}(E)=0.\]

This forms a subset of the numerical wall $\Upsilon_{E,F}$. The tilt wall is defined similarly but instead of using Bridgeland stability condition it is used the tilt stability function.

In \cite[Chapters 3,4,6]{JM}, the authors define and study a number of distinguished curves from which we can extract information about the numerical tilt and $\lambda$-walls by studying their relation with these curves. For the results presented here we will need to define the distinguished curve: \[ \Gamma_{E,s}=\{(\beta,\alpha)\in \mathbb{H}| \tau_{\beta,\alpha,s}(E)=0\}.\]

One important consequence from \cite[Chapter 6]{JM} is that numerical $\lambda$-walls for $E \in \KnumX$ have to cross $\Gamma_{E,s}$ at a horizontal point when $s=\frac{1}{3}$, the same relation that numerical tilt walls have with the distinguished curve $\theta_{E}=\{(\beta,\alpha)\in \mathbb{H}|\rho_{\beta,\alpha}(E)=0\}$. 

For the later parts of the paper it will be important to keep track of the orientation of the numerical walls. This means keeping track of which points in $\mathbb{H}$ an object $F$ destabilizes (numerically) another object $E$ or if it does not affect its stability.

 \begin{definition}
 Let $E,F \in \KnumX$. We will define the \emph{inside} of the numerical wall $\Upsilon_{E,F}$ as the subset where $f_{\sigma}(E,F)<0$. Moreover, the \emph{outside} of the wall $\Upsilon_{E,F}$ is the set of points where $f_{\sigma}(E,F)>0$.
 \end{definition}
 
 In particular, if $E,F$ are in $\mathcal{A}^{\beta,\alpha}$ and $s\in \mathbb{R}_{>0}$ then:  $(\beta,\alpha)\in\mathbb{H}$ is inside (outside) of the $\lambda$-wall in $\Upsilon_{E,F}$ if and only if $\lambda_{\beta,\alpha,s}(E)<(>)\lambda_{\beta,\alpha,s}(F)$.

\begin{example}
Let $X$ be the smooth quadric $Q_3$. This threefold was shown to satisfy the generalized Bogomolov inequality by Schmidt \cite{Sch2}. 

Using the results in \cite{Ott}, we can define the spinor bundle $S$ over $Q_3$ as the pullback of the universal bundle of the Grassmanian by the natural map $s: Q_3 \rightarrow Gr(2^3-1,2^{4}-1)$. This bundle is almost self-dual satisfying $S^\ast=S(1)$.

Comparing this to \cite{Sch2} we can see that the author is using $S^\ast$ as his spinor bundle when defining the quiver region, so it is important to keep this in mind when comparing our results and that one from \cite{Sch2}. In \cite[Theorem 2.1]{Ott}, the author proves that spinor bundles are $\mu$-stable and in \cite[Lemma 4.6]{Sch2} it is proven that $S^\ast(-1)[1]$ is tilt stable at the line $\beta=0$ of the upper-half plane $\mathbb{H}$ of tilt stability conditions. To our use here, we will need to translate the upper half-plane by tensoring by $\mathcal{O}_{Q_3}(-1)$ so that we will study the tilt stability of $S^{\ast}(-2)[1]$. 

Moreover, applying the numerical conditions of \cite[Theorem 3.3]{Sch1} and \cite{AMac} we can see that every numerical tilt wall for the object $S^{\ast}(-2)[1]$ has to cross the hyperbola $\rho_{\beta,\alpha}(S^{\ast}(-2)[1])=0$ at their apex and can never cross the numerical wall $\beta=\mu(S^{\ast}(-2)[1])$. So that we have the following Figure \ref{TiltStabSpinor}:

\begin{figure}[htp]
    \centering
    \includegraphics[width=10cm]{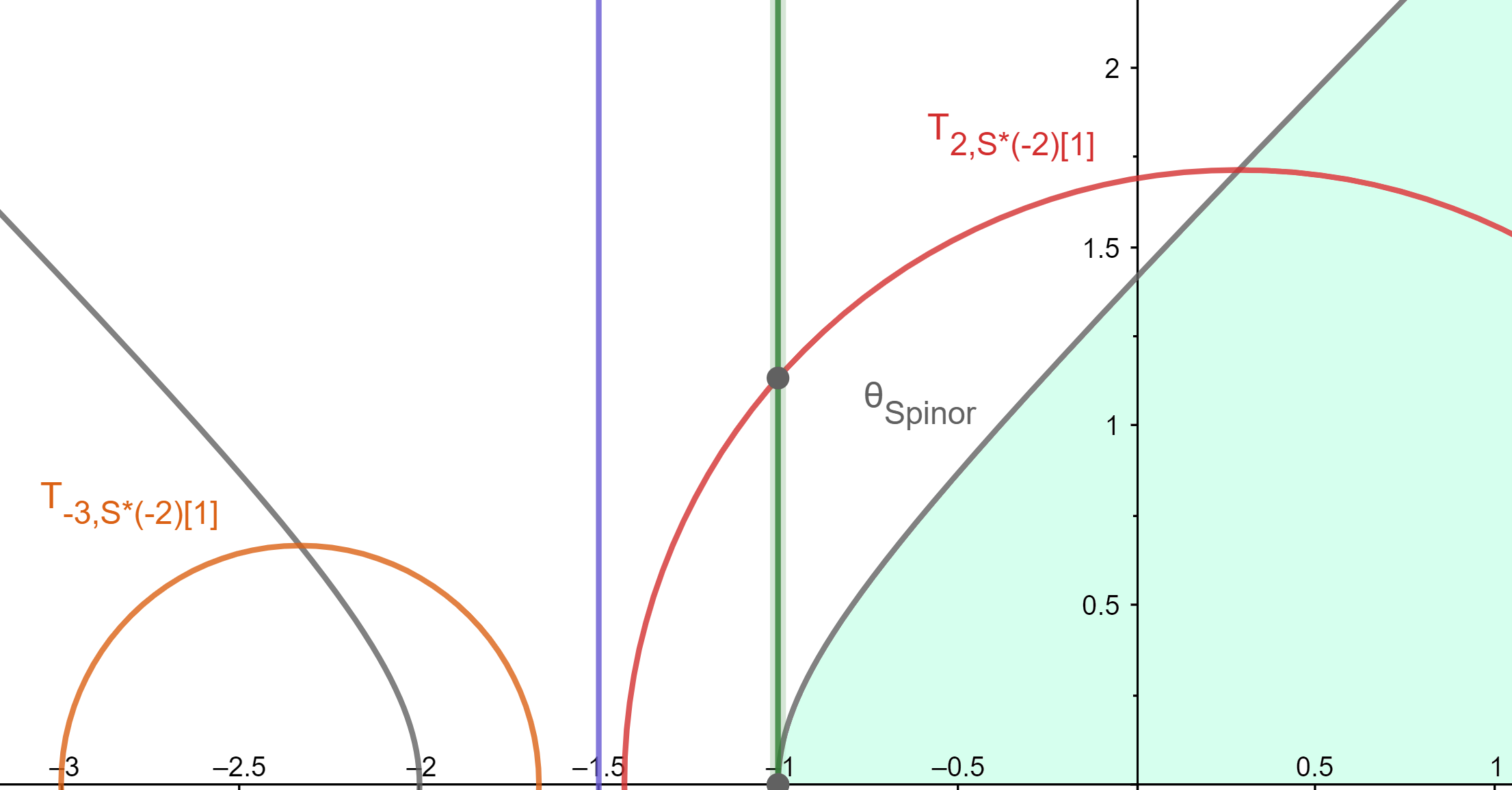}
    \caption{An example of two numerical tilt-walls: The first with respect to $S^\ast(-2)[1]$ and $\mathcal{O}_{Q_3}(2)$, $T_{2,S^\ast(-2)[1]}$, and the second to $S^\ast(-2)[1]$ and $\mathcal{O}_{Q_3}(-3)$, $T_{-3,S^\ast(-2)[1]}$. The region, in green, representing the points where $S^\ast(-2)[2] \in \mathcal{A}^{\beta,\alpha}$.}\label{TiltStabSpinor}
\end{figure}

Since $S^\ast(-2)[1]$ is tilt stable at $\beta=-1$, we can see that no numerical wall to the right-hand side of $\beta=-3/2=\mu(S^{\ast}(-2))$ can be an actual destabilizing tilt wall. So that $S^\ast(-2)[1] \in \mathcal{F}_{\beta,\alpha}$ if $(\beta,\alpha)$ satisfies $p_{\beta,\alpha}(S^\ast(-2)[1])\leq0$, i.e. the light green region in Figure \ref{TiltStabSpinor}.
\end{example}

\section{Quiver Stability}\label{Sec-QuiSt}

Another important notion of stability we will use is that of quiver stability as defined by King in \cite{King}, the main results in the paper come from the interplay of these two notions of stability.

A \emph{quiver} is a collection $Q=(Q_0,Q_1,s,t)$, where $Q_0$ is the set of vertices, $Q_1$ is the set of arrows between the vertices and $s,t:Q_1\rightarrow Q_0$ are the starter and terminal ends of the arrows, respectively. Our concern in this setting is with finite quivers, that means $Q_0$ and $Q_1$ are finite sets. 

The path algebra of $Q$ is an associative $k$-algebra $kQ$ of the quiver $Q$ with a basis given by paths in $Q$ and a product given by the formal concatenation of elements in $Q_1$ such that the multiplication satisfies $p \cdot q=0$ if $s(p)\neq t(q)$. Extending the functions $s,t$ to $kQ$, we define a relation for the quiver $Q$ to be any admissible ideal $J$ in $kQ$, that is, $J$ is an ideal generated by a finite set of $k$-linear combination of paths with the same starter and terminal ends. A quiver $Q$ with relations $J$ will be denoted by $(Q,J)$.

\begin{example}\label{Ex-Kronecker}
The $n$-Kronecker quiver $K_n$ is defined by $Q_0$ with two elements $p,q$ and $n$ arrows from $p$ to $q$ in $Q_1$. 
\end{example}

\begin{example}\label{Ex-Beilinson}
The Beilinson quiver $B_n$ is the realization of the necessary conditions to obtain a $n+1$-step complex in terms of consecutive sums of linear sheaves in the canonical helix in $\mathbb{P}^n$. That is, a quiver with $n+1$ vertices $0,...,n+1$ such that there are $n+1$ arrows between $i$ and $i+1$, for each $i$, named $\{x_j^i| j \in \{0,...,n\}\}$ satisfying the relation $x_j^{i+1}x_k^{i}-x_k^{i+1}x_j^{i}=0$, for every $i,j,k$.
\end{example}

A \emph{representation} of the quiver $Q$ is given by a graded $k$-vector space $V=\bigoplus_{i \in Q_0} V_i$ and $(f_h)_{h \in Q_1} \in \bigoplus_{h \in Q_1}\Hom(V_{s(h)},V_{t(h)})$. A morphism between two representations $(V,(f_h)_{h \in Q_1})$ and $(W,(g_h)_{h \in Q_1})$ is a collection of linear transformations $T:V_i \rightarrow W_i$ for each $i \in Q_0$, such that $T_{t(a)}\circ f_{a}=g_a \circ T_{s(a)}$. Moreover, if $(Q,J)$ is a quiver with relation, it is said that a representation of $Q$, $(V,(f_h)_{h \in Q_1})$, satisfies the relation $J$ if the compositions induced by the elements in $J$ are equal to $0$ as linear transformations.

With that definition it is possible to define the categories $\Rep(Q)$ and its subcategory $\Rep(Q,J)$ of representations of $Q$ and representations satisfying the relation $J$, respectively. These categories are both abelian. A dimension vector $\underline{\dim}:\Rep(Q,J) \rightarrow \mathbb{Z}^{Q_0}$ is well defined by letting $\underline{\dim}((V,(f_h)_{h \in Q_1}))=(\dim(V_i))_{i \in Q_0}$.

In \cite{King}, King explored the idea of defining stable representations to produce a GIT quotient in terms of a linear relation on the dimension vector of subrepresentations. Fix $\underline\theta \in \mathbb{R}^{Q_0}$, a representation $(V,(f_h)_{h \in Q_1})$ is said to be $\theta$-(semi)stable if \[\underline\theta \cdot \underline{\dim}(V,(f_h)_{h \in Q_1}):= \sum_{i \in Q_0}\theta_i \dim(V_i)=0\] and for every subrepresentation $(W,(g_h)_{h \in Q_1})$ of $(V,(f_h)_{h \in Q_1})$, \[\underline\theta \cdot \underline{\dim}(W,(g_h)_{h \in Q_1})\leq(<) 0.\] 

With this notion of stability is possible to define the moduli space of (semi)stable representations of a quiver with relations $(Q,J)$ with given dimension vector $\underline d \in \mathbb{Z}^{Q_0}$ with respect to $\underline\theta$ as $\mathcal{M}^{\underline\theta-s}_{Q,J}(\underline d)$($\mathcal{M}^{\underline\theta-ss}_{Q,J}(\underline d)$), where we identify the (semi)stable objects that have the same Jordan-H\"older filtration in $\Rep(Q,J)$. When $J=0$ we will suppress the subscript $J$ from the notation. Further information can be found at \cite{King, Andre}.

One way to construct the moduli spaces $\mathcal{M}^{\underline\theta-s}_{Q,J}(\underline d)$ and $\mathcal{M}^{\underline\theta-ss}_{Q,J}(\underline d)$ is by constructing a GIT quotient of the space $R_{\underline d}=\bigoplus_{h \in Q_1}\Hom(V_{s(h)},V_{t(h)})$ by the algebraic group of product of nonsingular linear transformations $\Gl_{\underline d}=\prod_{i \in Q_0}\Gl(d_i)$, acting via conjugation, which descends to $P\Gl_{\underline d}=\Gl_{\underline d}/k^\ast(\Id)_{i \in Q_0}$.

\begin{rmk}\cite[Remark 4.1.15]{Andre}
For a given $\underline\theta$ and dimension vector $d$, $\mathcal{M}^{\underline\theta-s}_{Q}(\underline d)$ is a smooth variety of dimension $1-\chi(d)= \Sigma_{h \in Q_1}d_{s(h)}d_{t(h)}-\Sigma_{i \in Q_0}d_i^2+1$. For the case of the Kronecker quiver $K_c$ with dimension vector $d=(c,2c+2)$ or $d=(2c+2,c)$, for some $c \in \mathbb{Z}_{\geq 0}$, it means that $\mathcal{M}^{\underline\theta- s}_{Q}(\underline d)$ has dimension $3c^2-3$.
\end{rmk}

\begin{rmk}\cite[Remark 4.1.15]{Andre}
By applying King's GIT-construction of the moduli space of semistable representations $\mathcal{M}^{\underline\theta-ss}_{Q,J}(\underline d)$, we have that it is a projective variety and the loci of stable representations $\mathcal{M}^{\underline\theta-s}_{Q,J}(\underline d)$ is an open subset of $\mathcal{M}^{\underline\theta-ss}_{Q,J}(\underline d)$. When $\underline d$ is $\underline\theta$-coprime, which occurs if $\underline\theta \cdot \underline{d'}\neq 0$ for every $\underline{d'}<\underline d$, then $\mathcal{M}^{\underline\theta-ss}_{Q,J}(\underline d)=\mathcal{M}^{\underline\theta-s}_{Q,J}(\underline d)$.
\end{rmk}

\section{Exceptional Collections}\label{Sec-ExcCo}

\subsection{Definitions}

We follow the notation used in \cite{Mac} for exceptional collections and related concepts. A structure of linear triangulated category in $\Db$ is such that for any $A,B \in \Db$ we have the $\mathbb{Z}$-graded vector space \[ \Hom^\bullet_{\Db}(A,B)=\underset{i\in \mathbb{Z}}{\bigoplus}\Hom^{i}_{\Db}(A,B)=\underset{i \in \mathbb{Z}}{\bigoplus}\Hom_{\Db}(A,B[i]).\] 
For any $\mathbb{Z}$-graded vector space $V^\bullet$ and $E \in \Db$, let $V^\bullet \otimes E:=\underset{i \in \mathbb{Z}}{\bigoplus}V^i \otimes E$, where $V^i \otimes E$ can be understood as the direct sum $\dim V^i$ copies of $E$. The dual of a $\mathbb{Z}$-graded vector space $V^\bullet$ is $V^{\bullet\ast}$ defined as $(V^{\bullet\ast})^i:=(V^{-i})^\ast$ so that the dual of an object $V^\bullet \otimes E \in \Db$ is $(V^{\bullet\ast})\otimes E^\vee$, with $E^\vee=R\operatorname{\mathcal{H}om}(E,\mathcal{O}_X)$.

\begin{definition}
An object $E$ in $\Db$ is called \emph{exceptional} if $\Hom^\bullet(E,E)=\mathbb{C}$. A collection $\mathcal{E}:=\{E_0,...,E_n\}$ of exceptional objects is called \emph{exceptional} if satisfies $\Hom^\bullet(E_i,E_j)=0$ if $i>j$. An exceptional collection can have other properties such as:
\begin{itemize}
    \item \emph{Strong}: if $\Hom^k(E_i,E_j)=0$ for all $i,j$ and $k\neq 0$,
    \item \emph{Ext}: if $\Hom^{\leq 0}(E_i,E_j)=0$ for all $i\neq j$,
    \item \emph{Full}: if the category generated by $\mathcal{E}$ via shifts and extensions is $\Db$.
\end{itemize}
\end{definition}

 One constructs an Ext-exceptional collection from a strong exceptional collection by performing the following trick: starting with a strong exceptional collection $\mathcal{E}=\{E_0,E_1,E_2,E_3\}$ then $S(E):=\{E_0[3],E_1[2],E_2[1],E_3\}$ is an Ext-exceptional collection. This allows us to reduce the search for possible Ext-exceptional collections by using strong exceptional collections.

\begin{proposition}\label{Theorem-Macri}\cite[Lemma 3.14 and Lemma 3.16]{Mac}
Let $\{E_0,...,E_n\}$ be a full  Ext-exceptional collection in $\Db$ then the category generated by extensions  $\langle E_0,...,E_n\rangle$ is the heart of a bounded $t$-structure. Assume that $(Z,\mathcal{P})$ is a stability condition and $E_0,...,E_n$ are all in $\mathcal{P}((\phi,\phi+1])$ for some $\phi \in \mathbb{R}$, then $\langle E_0,...,E_n\rangle=\mathcal{P}((\phi,\phi+1])$ and each $E_i$ is stable.
\end{proposition}

Theorem \ref{Theorem-Macri} leads to the following definition used to determine the regions of $\mathbb{H}$ we are interested in. 

\begin{definition}\label{HplaneCon}
A full Ext-exceptional collection $\{E_0,...,E_n\}$ satisfies the \emph{upper-half plane condition} for a stability condition $(Z,\mathcal{P})$ if there exists a $\phi \in \mathbb{R}$ such that $\langle E_0,...,E_n\rangle=\mathcal{P}((\phi,\phi+1])$.
\end{definition}

As a consequence of Theorem \ref{Theorem-Macri}, Definition \ref{HplaneCon} is equivalent to the notion of $\sigma$-exceptional collection introduced in \cite[Definition 3.19]{DK}.

\begin{rmk}\label{Ex-phiRotation}
Let $\mathcal{E}=\{E_0,E_1,E_2,E_3\}$ be a full Ext-exceptional collection and $\mathcal{C}=\langle\mathcal{E}\rangle$ the heart of bounded $t$-structure generated by $\mathcal{E}$. If $\sigma_{\beta,\alpha,s}=(Z_{\beta,\alpha,s},\mathcal{A}^{\beta,\alpha})$ is a stability condition in $X$ such that $\mathcal{C}\subset D^{\beta,\alpha}:=\langle\mathcal{A}^{\beta,\alpha},\mathcal{A}^{\beta,\alpha}[1]\rangle$ then we can define the $\mathbb{C}$-slope of a semistable object $E \in D^{\beta,\alpha}$ as the unique $\psi \in (0,2]$ with $Z_{\beta,\alpha,s}(E)=r\cdot e^{\pi\psi i}$ and $r \in \mathbb{R}_{>0}$. Moreover, we can rephrase Definition \ref{HplaneCon} as the existence of a $\phi \in (0,1]$ where the upper-half plane $\UpH$ rotated by $(\phi \pi)$-degrees contains all the complex numbers $Z_{\beta,\alpha,s}(E_i)$, for all $i$.
\end{rmk}

\begin{figure}[htp]
    \centering
    \includegraphics[width=10cm]{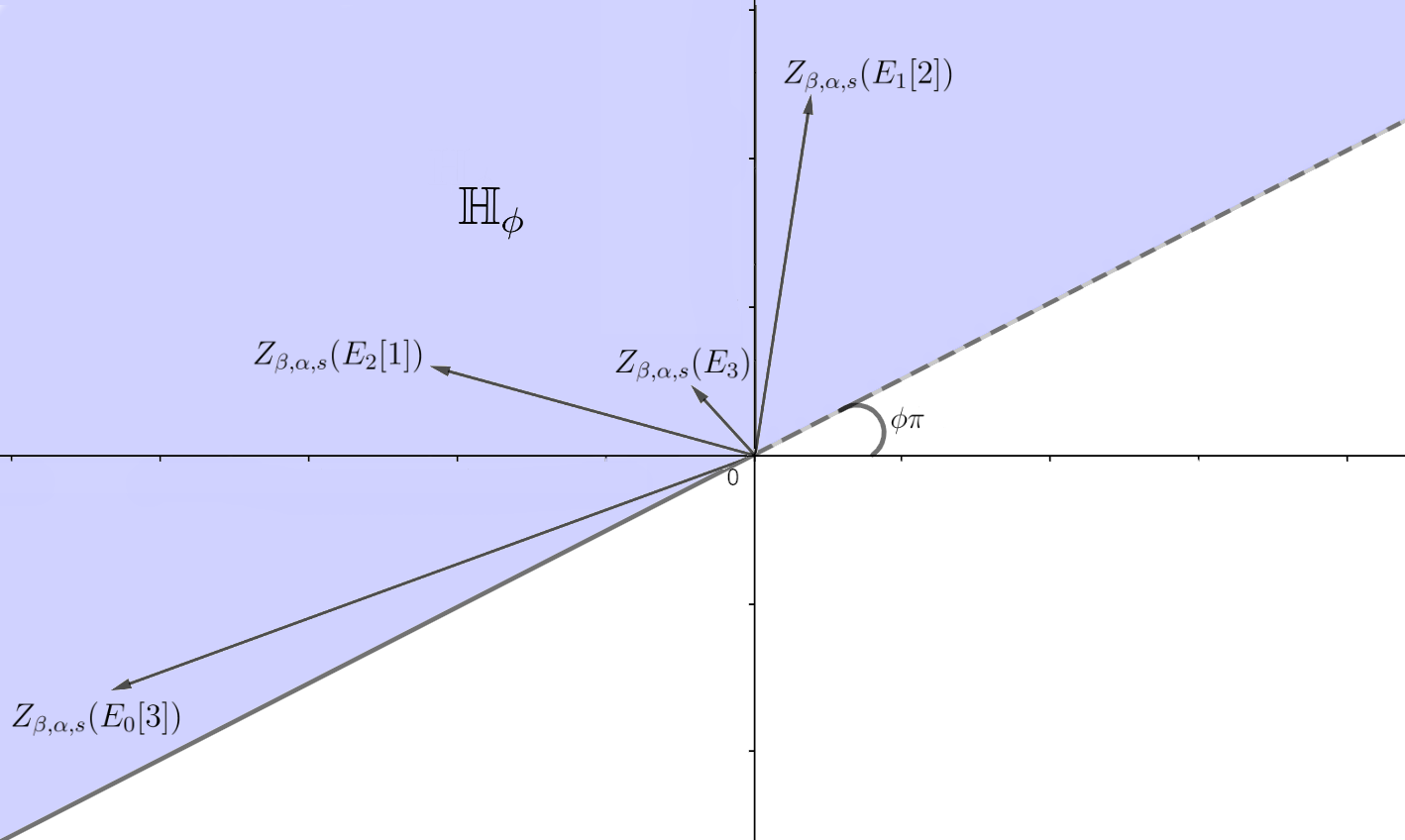}
    \caption{The purple region is $\bar{\mathbb{H}}_\phi$}\label{Half-plane condition}
\end{figure}

We will denote by $\UpH_\phi$ the upper-half plane obtained from rotating $\UpH$ by $(\phi\pi)$-degrees. In the previous example it is clear that if $E$ is $\sigma_{\beta,\alpha,s}$-semistable with $\mathbb{C}$-slope $\phi \in (0,2]$ then either $E$ or $E[1]$ is in $\mathcal{C}$.

\begin{rmk}
Note that the choice of $\phi$ in Remark \ref{Ex-phiRotation} is not  canonical and we can have an interval of angles satisfying this condition. Furthermore, let $\mathcal{E}$ be a full Ext-exceptional collection satisfying the upper-half plane condition for all $\phi \in (\psi,\psi')\subset (0,1]$ and fix  $(\beta,\alpha) \in \mathbb{H}$. Then there is no semistable object $E \in \mathcal{A}^{\beta,\alpha}$ satisfying $Z_{\beta,\alpha,s}(E)=r_E \cdot e^{i\pi\gamma}$ with $\gamma \in (\psi,\psi')$. This is due to the fact that if such an $E$ were to exist, then it would belong to $\mathcal{P}_{\beta,\alpha}((\psi,\psi+1])$ but not to $\mathcal{P}_{\beta,\alpha}((\psi',\psi'+1])$. However, both of these categories are equal to $\mathcal{C}$.    
\end{rmk}

Moreover, being able to generate an abelian category using an Ext-exceptional collection is a very useful tool to produce and classify the numerical $\lambda$-walls for the object $F \in \langle\mathcal{E}\rangle$. 

To do that we use the structure of the admissible triangulated subcategory \[\Tr(E_{k+1},...,E_n)\xrightarrow{i_{k\ast}} \Dbal(X) \] as in \cite[Lemma 3.14]{Mac} with its right orthogonal subcategory $\Tr(E_0,...,E_k)$, consequence of \cite[Lemma 6.1]{Bon1}. Associated to it, there exists a functor $\Db \xrightarrow{j_k^\ast}\Tr(E_0,...,E_k)$. By the same argument therein, it is possible to see that the right adjoint functor $i_k^!$ to $i_{k\ast}$ decomposes the object $E \in \langle\mathcal{E}\rangle$ in the exact sequence \begin{equation}\label{Truncation Functor}
    0 \rightarrow i_k^!(E) \rightarrow E  \rightarrow j_k^\ast(E) \rightarrow 0.
\end{equation} 

Denote the functors $i_k^!$ and $j_k^\ast$ by the truncation functors $\tau_{>k}$ and $\tau_{\leq k}$, respectively.

\section{Quiver Regions}\label{Sec-QuivR}

A quiver region is where a Bridgeland stability condition behaves as a quiver stability, in the sense of Section \ref{Sec-QuiSt}. The goal of this section is to lay the foundations on how to obtain the quiver regions. We start by defining what a quiver region is meant to be, and providing a simple way of calculating them. This allows for us to then define the specific quiver regions on which we will work on.

In this section we assume that $X$ is a smooth projective variety such that $\Stab(X)$ is non empty.

\begin{definition}\label{DEF-Excp-region}
 A subset $\widetilde{R_\mathcal{E}}$ of $\Stab(X)$ is called a \emph{quiver region} for a full Ext-exceptional collection $\mathcal{E}=\{E_0,...,E_n\}$ if $\mathcal{E}$ satisfies the upper-half plane condition (see Definition \ref{HplaneCon}) for every stability condition $(Z,\mathcal{P})\in \widetilde{R_\mathcal{E}}$. When $\dim(X)=3$, $\mathcal{G}$ is a strong exceptional collection and we fix $s>0$, denote by \[R_\mathcal{G}=\widetilde{R_{S(\mathcal{\mathcal{G}})}} \cap \mathbb{H} \subset \Stab(X),\]
where we identify $\mathbb{H}$ with the image of $\phi_s:\mathbb{H}\rightarrow \Stab(X)$.
\end{definition}

\begin{rmk}
 The number $\phi_{\beta,\alpha} \in (0,1]$ such that $\mathcal{E} \subset \mathcal{P}_{\beta,\alpha}((\phi_{\beta,\alpha},\phi_{\beta,\alpha}+1]))$ can vary when choosing $(\beta,\alpha) \in R_\mathcal{E}$. In addition, we do not assume that a quiver region is maximal in the sense that we are not assuming that these are the only points in $\mathbb{H}$ satisfying  $\mathcal{E} \subset \mathcal{P}_{\beta,\alpha}((\phi_{\beta,\alpha},\phi_{\beta,\alpha}+1]))$ for some  $\phi_{\beta,\alpha} \in (0,1]$.
\end{rmk}

 Conceptually, the existence of non-empty quiver regions for a full Ext-exceptional collection can be regarded as a discretization of the categories $\mathcal{A}^{\beta,\alpha}$, which vary continuously. Following Mu's definition \cite[Definition 5.2]{Mu}, and the relation between the category generated by exceptional collections and quivers expressed in \cite{Mac} and \cite{Bert}, we define a dimension in the category generated by a strong exceptional collection of sheaves. 
 
 \begin{definition}
 Let $\mathcal{E}=\{E_0,...,E_n\}$ be a full strong exceptional collection of sheaves in $\Db$. The \emph{dimension vector} for an object $F \in \langle S(\mathcal{E})\rangle$ is $\dim_\mathcal{E}(F)=[a_0,...,a_n]$, where $\ch(F)=\Sigma_{i=0}^n(-1)^{n+1-i}a_i\ch(E_i)$ and $a_i \in \mathbb{Z}$.
 \end{definition}
 
 There exists a partial ordering for the dimension of objects in $\Db$ given by $\dim_\mathcal{E}(F)=[b_0,...,b_n] \leq \dim_\mathcal{E}(E)=[a_0,...,a_n]$ if and only if $b_i\leq a_i$, for each $i$. When clear from the context, we will omit the subscript $\mathcal{E}$.
 
 \begin{lemma}\label{Lemma-Dimension}
 Let $\mathcal{E}=\{E_0,...,E_n\}$ be a full strong exceptional collection of sheaves in $\Db$ and $S(\mathcal{E})$ its shift. Then for $E \in \langle S(\mathcal{E})\rangle$:
 \begin{itemize}
     \item[(a)] $\dim_\mathcal{E}(E)=[a_0,...,a_n]$ with $a_i\geq 0$ for all $i$,
     \item[(b)] If \[ 0 \rightarrow F \rightarrow E \rightarrow G \rightarrow 0\] is an exact sequence in $\langle S(\mathcal{E})\rangle$ then $\dim_\mathcal{E}(F)+\dim_\mathcal{E}(G) \leq \dim_\mathcal{E}(E)$.
 \end{itemize}
 \end{lemma}
 
 \begin{proof}
 The first item in the Lemma is clear from the definition, as objects in $\langle S(\mathcal{E})\rangle$ are obtained by isomorphism classes of extensions of the objects in $S(\mathcal{E})$. Item $(b)$ is a consequence of the additive behaviour of the Chern character in $\KnumX$.
 \end{proof}

  The next proposition is responsible for calculating quiver regions when applied to $R_\mathcal{E}$ in $\mathbb{H}$ for a fixed $s>0$. We will describe the regions we will work with in the following examples. 
 
 \begin{proposition}\label{Prop-Upper-Half-plane condition}
  Let $\mathcal{E}=\{E_0,...,E_n\}$ be a full Ext-exceptional collection in $\Db$ such that $E_i \in \mathcal{P}((\phi,\phi+2])$, for some $\phi \in \mathbb{R}$ and some slicing $\mathcal{P}$. Then $\mathcal{E}$ satisfies the upper-half plane condition with respect to a Bridgeland stability condition $\sigma=(Z,\mathcal{P})$ if and only if there exists $r\in\{0,...,n\}$ such that:
  \begin{itemize}
      \item[($\ast$)]   $(Z,\mathcal{P})$ is inside the walls $\Upsilon_{E_r[-1],E_j[-1]}$, for all $j \in \{j | E_j \in \mathcal{P}((\phi+1,\phi+2])\}$, outside the walls $\Upsilon_{E_r[-1],E_k}$ for all $k \in \{k | E_k \in \mathcal{P}((\phi,\phi+1])\}$.
  \end{itemize}
  
 \end{proposition}
 
 In this proof, the notion of slope used is the $\mathbb{C}$-slope defined in Remark \ref{Ex-phiRotation} and we will assume $\phi=0$ to make the presentation clearer. To see that this assumption does not limit our conclusions, consider a new formulation of $(\ast)$: there exists $r$ such that $(Z,P)$ is inside every wall $\Upsilon_{E_r,E_i}$ for every $i \in\{0,...,n\}$. This equivalent to $(\ast)$ and it does not depend on $\phi$. This equivalence is crucial for the numerical results we are about to share because we do not need to figure out if the objects are in $\mathcal{A}^{\beta,\alpha}$ or $\mathcal{A}^{\beta,\alpha}[1]$.
 
 \begin{proof}
 Assume that $\mathcal{E}$ satisfies the upper half-plane condition for some $\psi \in [0,1)$. If $\psi=0$ this imply that $E_i \in \mathcal{P}((0,1])$ for all $i$ and we can always choose $r$ such that $\phi_{\sigma}(E_r)$ is the smallest slope within the set $\{\phi_{\sigma}(E_i)\}$. Then $(Z,\mathcal{P})$ with this $r$ satisfies condition $(\ast)$. 
 
 Assume now that $\psi \neq 0$ and that there exists $E_i \in \mathcal{P}((1,2])$, if this was not the case then we could apply the previous case. Ordering the $\mathbb{C}$-slopes of $Z(E_i)$ we can choose $r$ such that $Z(E_r)$ has a $\mathbb{C}$-slope greater than or equal to all of the other $Z(E_i)$. Therefore, $r$ must satisfy the condition $(\ast)$ in the theorem as the $Z(E_i)$ are all constrained by a upper-plane rotated $\psi\cdot\pi$-degrees. 
 
 For the reverse implication let $r$ be such that it satisfies condition $(\ast)$ for $\sigma=(Z,\mathcal{P})$. Then we can consider $\psi$ as the $\mathbb{C}$-slope of $Z(E_r[-1])$ and prove that $\mathcal{E}$ satisfies the upper half-plane condition for $\psi$. This is done by observing that we can place every exceptional object in $\mathcal{E}$ into $\mathcal{P}((\psi,\psi+1])$: either it already is in $\mathcal{P}((\psi,\psi+1])$ or in $\mathcal{P}((\psi,\psi+1])[1]$, and divide the upper half-plane $\UpH$ by the vector $Z(E_r[-1])$ making it such that the right-hand side we have the vectors $Z(E_j[-1])$ with slope less than the slope of $Z(E_r[-1])$, similarly for the vectors $Z(E_k)$ the slope is greater than the slope of $Z(E_r[-1])$. In other words, the vectors $Z(E_i)$, for all $i$, are bounded by the upper half-plane defined by $Z(E_k)$.
 \end{proof}

\begin{example}\label{Def-R1}
Let $X=\mathbb{P}^3$ and $S(\mathcal{E}_1)=\{\opn(-2)[3],\opn(-1)[2],\opn[1],\opn(1)\}$ a strip of the canonical helix of $\mathbb{P}^3$ shifted by the necessary degrees so that $S(\mathcal{E}_1)$ is a full Ext-exceptional collection. We will fix $s=1/3$, since $s>1/3$ does not change the existence of walls in $\mathbb{H}$, it just dilates the wall in the $\alpha$ direction, see \cite{JM}. The region of the points $(\beta,\alpha) \in \mathbb{H}$ where $\mathcal{E}$ is in $\langle\mathcal{A}^{\beta,\alpha},\mathcal{A}^{\beta,\alpha}[1]\rangle$ is displayed in Figure \ref{Graph-Category condition}. This is the necessary condition to apply Proposition \ref{Prop-Upper-Half-plane condition}.

\begin{figure}[htp]
    \centering
    \includegraphics[width=10cm]{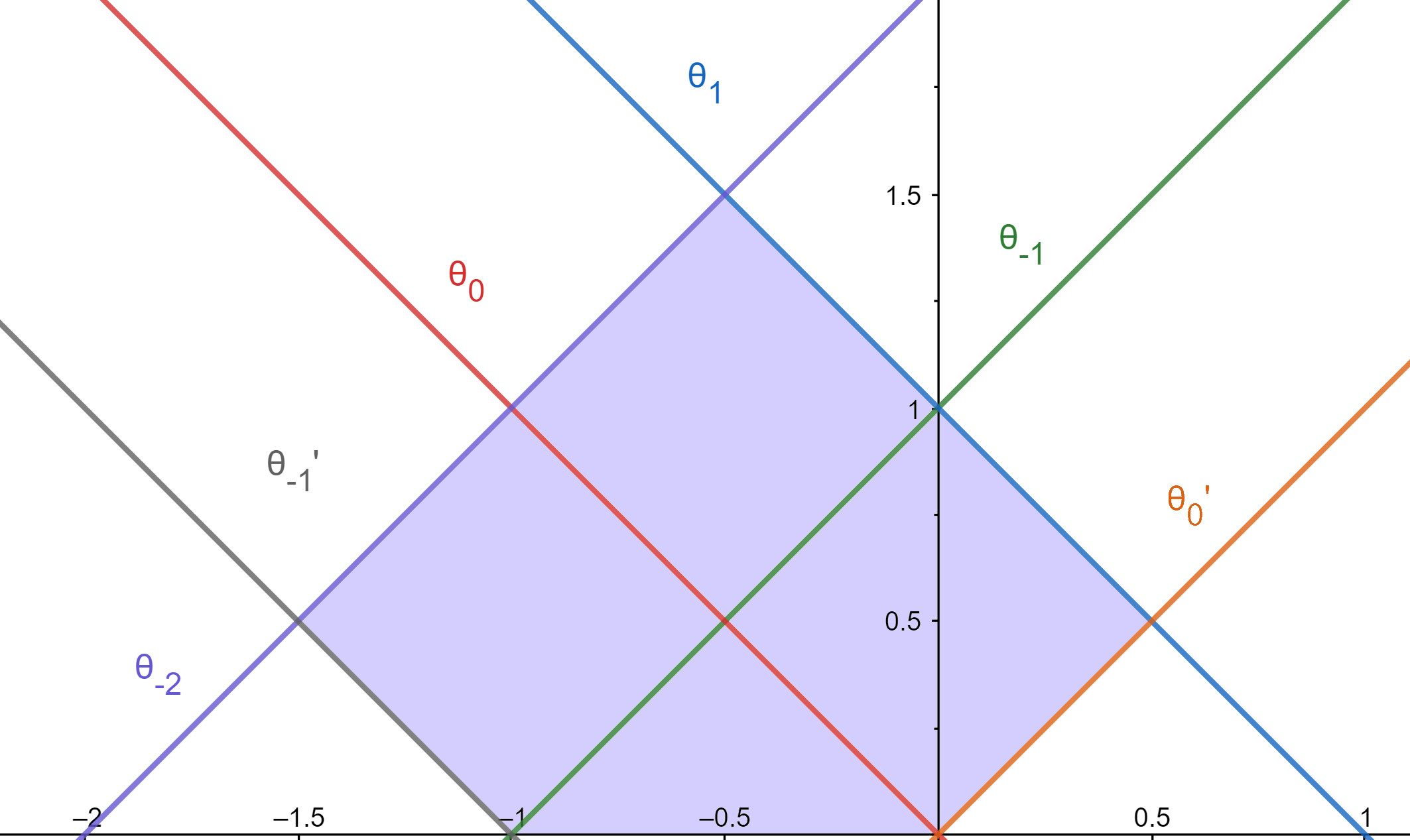}
    \caption{Region of $\mathbb{H}$ where the exceptional objects $E_{2-k}=\opn(-k)[k+1]$ satisfy $E_i \in \langle \mathcal{A}^{\beta,\alpha},\mathcal{A}^{\beta,\alpha}[1]\rangle$ and $\theta_i=\theta_{\mathcal{O}_{\mathbb{P}^3}(i)}$.}\label{Graph-Category condition}
\end{figure}        

Figure \ref{Upper-half P3} is done using the numerical condition $(\ast)$ in Proposition \ref{Prop-Upper-Half-plane condition} for the collection $S(\mathcal{E}_1)$.

\begin{figure}[htp]
    \centering
    \includegraphics[width=10cm]{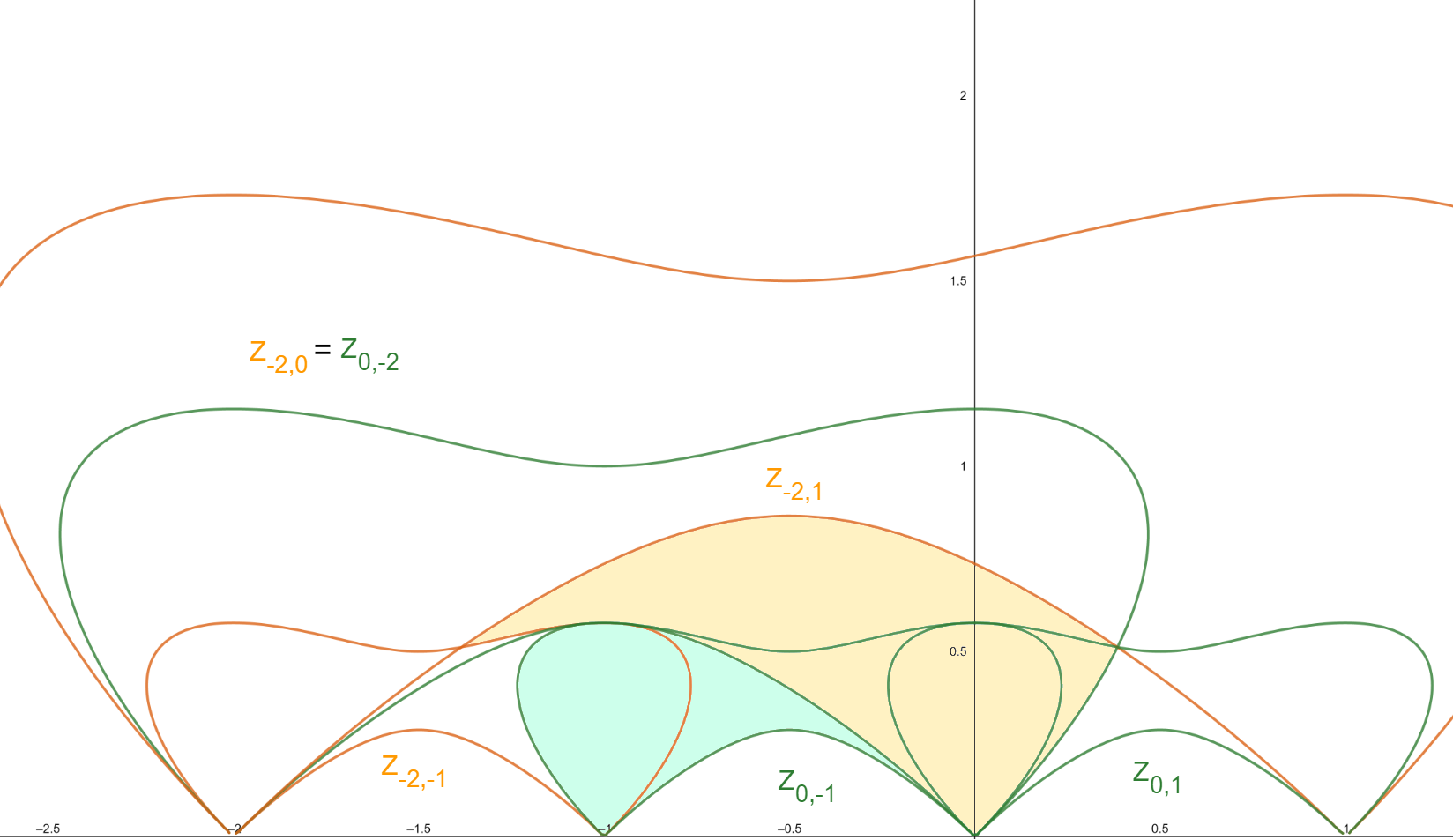}
    \caption{The wall $Z_{i,j}$ represents the $\lambda$-wall in $\Upsilon_{\opn(i)[1-i],\opn(j)[i-j]}$, with the yellow and green walls representing the application of Proposition \ref{Prop-Upper-Half-plane condition} to $i=-2$ and $i=0$, respectively. }\label{Upper-half P3}
\end{figure}

Concluding, we draw the region $R_1$ as the intersection of the regions obtained by Figures \ref{Graph-Category condition} and \ref{Upper-half P3}.
\end{example}

\begin{example}\label{Def-Quadrica}
Consider now the smooth quadric $X=Q_3$. For the exceptional objects $\mathcal{O}_{Q_3}(i)$, the distinguished curves are the same as the ones defined for $\mathcal{O}_{\mathbb{P}^3}(i)$ because their $\lambda_{\beta,\alpha,\frac{1}{3}}$ and $\nu_{\beta,\alpha}$ slopes are the same. Combining Figures \ref{Graph-Category condition} and \ref{TiltStabSpinor}, we can see a region of $\mathbb{H}$ where the exceptional collection \[S(\mathcal{E}_2)=\{S^{\ast}(-2)[3],\mathcal{O}_{Q_3}(-1)[2],\mathcal{O}_{Q_3}[1],\mathcal{O}_{Q_3}(1)\}\] is contained in $\langle \mathcal{A}^{\beta,\alpha},\mathcal{A}^{\beta,\alpha}[1]\rangle$, denote it by $P$. To apply Proposition \ref{Prop-Upper-Half-plane condition} we just need to check condition $(\ast)$ and this is done in Figure \ref{Image- Quadric R_3}.

\begin{figure}[htp]
    \centering
    \includegraphics[width=10cm]{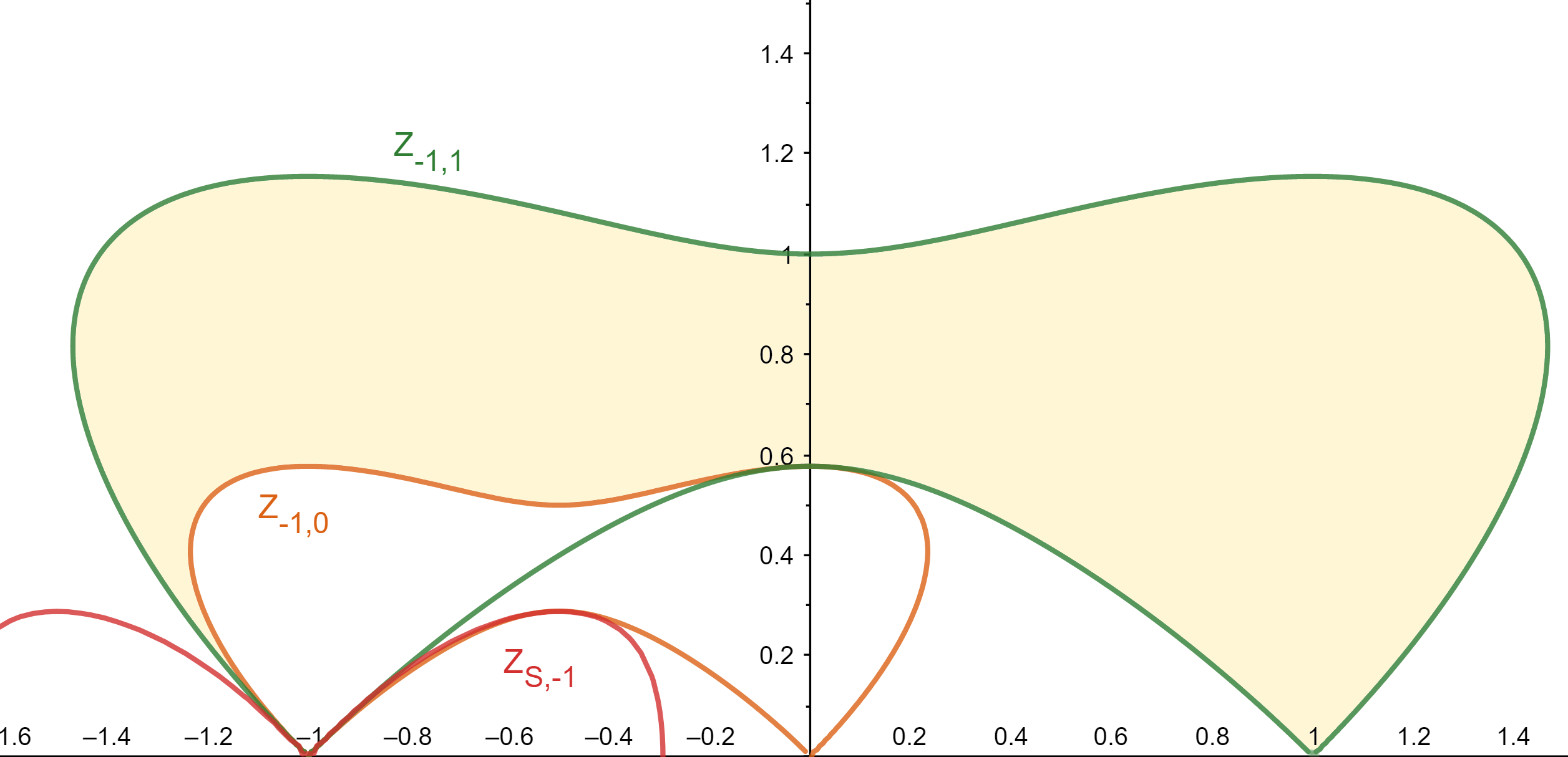}
    \caption{Same notation as the one used in Figure \ref{Upper-half P3} and \linebreak $Z_{S,-1}=\Upsilon_{S^{\ast}(-2)[3],\mathcal{O}_{Q_3}(-1)[2]}$.}\label{Image- Quadric R_3}
\end{figure}

We define the region $R_2$ as the intersection of the regions $P$ and the yellow region in Figure \ref{Image- Quadric R_3}.
\end{example}

\section{Technical Results and some Linear Algebra}\label{Sec-TecLin}

In this section we explore some technical results, which are important tools to deal with semistable objects inside the quiver region.

We keep using that $X$ is a projective smooth threefold with non-empty space of Bridgeland stability conditions and fix $\mathcal{E}=\{E_0,...,E_3\}$ to be a full strong exceptional collection in $\Db$, $\widetilde{R_{S(\mathcal{\mathcal{E}})}}$ its quiver region in $\Stab(X)$ and $R_\mathcal{E}=\widetilde{R_{S(\mathcal{\mathcal{E}})}}\cap \mathbb{H}$, the section of the geometric stability conditions with respect to some $s>0$.

We start with a reference to a correspondence between the moduli space of $\theta$-(semi)stable representations of a given quiver and the moduli space of Bridgeland (semi)stable objects. This gives the necessary background to relate Bridgeland stability and quiver stability.

\begin{theorem}\cite[Theorem 8.1]{Bert}
Let $S(\mathcal{E})$ be a full Ext-exceptional collection of a smooth projective threefold $X$, $(\beta,\alpha)$  a point in the quiver region $R_\mathcal{E}$. Then the moduli space $\mathcal{M}^{ss}_{\beta,\alpha}[a,b,c,d]$ of $\lambda_{\beta,\alpha,s}$-semistable objects with dimension vector $[a,b,c,d]$ is a projective variety. Furthermore, if we consider only the stable objects $\mathcal{M}^{s}_{\beta,\alpha}[a,b,c,d]$ then this space is a quasi-projective variety.
\end{theorem}

 In \cite{Bert}, the authors prove the theorem only for the case of exceptional collections on $\mathbb{P}^2$, but the proof is the same for any smooth projective variety with an exceptional collection with a nonempty quiver region.
    
 When either $a=b=0$, $d=c=0$ or $a=d=0$ the quiver obtained is the Kronecker quiver $K_n$ with $n=\dim\Hom(E_2,E_3)$, $\dim\Hom(E_0,E_1)$ or $\dim\Hom(E_1,E_2)$, respectively.

The next lemma determines what are the walls for a $2$-step complex and defines the determinant condition. This condition is responsible for relating Mumford's $\mu$-stability to Bridgeland stability.

\begin{lemma}\label{Lemma-Determinant Condition}
Let $K$ be an object in $\langle S(\mathcal{E})\rangle$ with dimension vector  $\ch(K)=(-1)^ia\ch(E_i)-(-1)^ib\ch(E_{i+1})$, for some $i\in\{0,1,2\}$. If there exists an actual wall with respect to $\sigma=(Z,\mathcal{P})\in \widetilde{R_\mathcal{E}}$ defined by

\begin{equation}\label{e1}
    0 \rightarrow F \rightarrow K \rightarrow G \rightarrow 0,
\end{equation}

\noindent then $\Upsilon_{F,K}=\Upsilon_{E_i,E_{i+1}}$. Furthermore, if every subobject $F$  of $K$ in $\langle S(\mathcal{E})\rangle$ with $\ch(F)=(-1)^ic\ch(E_i)-(-1)^id\ch(E_{i+1})$ satisfies $(a\cdot d-b \cdot c)(\geq) >0$ then $K$ is Bridgeland (semi)stable outside the curve determined by $\Upsilon_{E_i,E_{i+1}}$ and inside of $\widetilde{R_\mathcal{E}}$. In this case, we will say that $K$ satisfies the (semi)-determinant condition.

\end{lemma}

\begin{proof}

Assume that we have the actual wall determined by sequence (\ref{e1}) so that $F,K,G$ are all objects of $\mathcal{P}(\phi)$,  and by our hypothesis $K \in \langle S(\mathcal{E})\rangle$, therefore all of them are in $\langle S(\mathcal{E})\rangle$. Furthermore, since $\langle S(\mathcal{E}) \rangle = \mathcal{P}((\psi,\psi+1])$, we can use their $\mathbb{C}$-slope to calculate their stability with respect to $\sigma=(Z,\mathcal{P})$. This also implies the restrictions to the dimension vector of the subobject $F$ by Lemma \ref{Lemma-Dimension}.

Now, in order to prove the uniqueness and establish the orientation of the actual wall for $K$ we have to consider the defining equation of the actual wall and use the fact that the stability function is a group homomorphism $Z$ in $\Lambda$. The defining equation for the wall $\Upsilon_{F,K}$ is

\begin{equation*}
    f_\sigma(F,K)=0,
\end{equation*}
which is the same as
\begin{equation*}
   \frac{ c \cdot(-\Real(Z(E_{i})))-d\cdot(-\Real(Z(E_{i+1})))}{c\cdot\Imag(Z(E_{i}))-d\cdot\Imag(Z(E_{i+1}))}=\frac{a\cdot(-\Real(Z(E_{i})))-b\cdot(-\Real(Z(E_{i+1})))}{a\cdot\Imag(Z(E_{i}) )-b\cdot\Imag(Z(E_{i+1}))}
\end{equation*}
and this equation is the equivalent to
\begin{equation}
    (a\cdot d-b\cdot c)f_{\sigma}(E_i,E_{i+1})=0.
\end{equation}

The positivity (negativity) of the determinant $(a\cdot d-b \cdot c)$ provides the orientation of the wall $\Upsilon_{F,K}$, whether it will make $\arg Z(F)>\arg Z(K)$ or $\arg Z(F)<\arg(K)$.
\end{proof}

\begin{corollary}
Inside a quiver region $R_\mathcal{E}$ with respect to a strong exceptional collection $\mathcal{E}$, there exists at most one numerical $\lambda$-wall for a given $2$-step complex object $K \in \langle S(\mathcal{E})\rangle$.
\end{corollary}

Since we are mostly working with $3$-step complexes, it will be useful to generalize the calculations done in Lemma \ref{Lemma-Determinant Condition}. This now calculates the $\theta$ for relating the $\theta$-stability of representations of the Beilinson quiver to Bridgeland stability, in the $\mathbb{P}^3$ and $Q_3$ case.

\begin{lemma}\label{Lemma-Determinant-Calculation}
  For any $3$-step complex $L \in \Shi$ with dimension vector \[\ch(L)=-(-1)^ia\ch(E_{i-1})+(-1)^ib\ch(E_i)-(-1)^ic\ch(E_{i+1}),\] for $i=\{1,2\}$, the numerical walls of a stability $\sigma=(Z,\mathcal{P})\in R_\mathcal{E}$ for the object $L$ and $F\in\Shi$ have to go through the points where $f_\sigma(E_i,E_{i+1})=f_\sigma(E_i,E_{i-1})=0$, if it exists.
\end{lemma}
\begin{proof}
The proof is just an observation that for any $F \in \Shi$ subobject of $L$ with $\ch(F)=-(-1)^ia'\ch(E_{i-1})+(-1)^ib'\ch(E_i)-(-1)^ic'\ch(E_{i+1})$, we have a linear description of the defining equation for the $\lambda$-wall as
\begin{equation}
f_\sigma(F,L) =\begin{vmatrix}
f_\sigma(E_{2},E_{1}) &f_\sigma(E_{2},E_{0}) &f_\sigma(E_1,E_0)\\
a &b &c\\
a' &b' &c'
\end{vmatrix}
\end{equation}
Therefore if $\sigma=(Z,\mathcal{P})$ satisfies $f_\sigma(E_i,E_{i+1})=f_\sigma(E_i,E_{i-1})=0$ then $f_\sigma(F,L)=0$. 
\end{proof}

\begin{rmk}\label{rmk-lambwall}
When clear, in the threefold case for a fixed $s>0$, $L\in \KnumX$, we will use the notation $\Upsilon_i$ for the $\lambda$-walls defined by $\{(\beta,\alpha)\in \mathbb{H}| f_{\sigma_{\beta,\alpha,s}}(L,E_i[3-i])=0\}$ respectively for every $i=1,2,3$.
\end{rmk}

\begin{proposition}\label{Prop-Mustability}
Suppose that $\{E_0,...,E_3\}$ is a strong exceptional collection of $\mu$-stable sheaves. For a coherent sheaf $K$ with an exact resolution \[0 \rightarrow E_{i}^{\oplus a} \rightarrow E_{i+1}^{\oplus b} \rightarrow K \rightarrow 0\] for some $i\in\{0,...,2\}$ to satisfy the (semi) determinant condition it is sufficient that $K$ is a $\mu$-(semi)stable sheaf.
\end{proposition}

\begin{proof}
By the definition of $\langle S(\mathcal{E})\rangle$, it is clear that $K[n-i-1] \in \langle S(\mathcal{E})\rangle$ and suppose that $F[n-i-1]$ is a subobject of $K[n-i-1]$ in $\langle S(\mathcal{E})\rangle$. Using Lemma \ref{Lemma-Dimension} we know that $F[n-i-1]$ is quasi-isomorphic to $(E_{i}^{\oplus c} \rightarrow E_{i+1}^{\oplus d})[n-i-1]$ and from applying the cohomology functor $\mathcal{H}^{0}_{\Coh}$ to \begin{equation}\label{Eq-Det-mu}
     0 \rightarrow F[n-i-1] \rightarrow K[n-i-1] \rightarrow L[n-i-1] \rightarrow 0
\end{equation}
we conclude that $F[n-i-1]$ is a sheaf shifted by $[n-i-1]$. Therefore, we can view the exact sequence \eqref{Eq-Det-mu} as the following exact diagram in $\Coh$:

\centerline{\xymatrix{
&&&\mathcal{H}^{-1}(L) \ar@{^{(}->}[d]\\
 0 \ar[r] & E_i^{\oplus c} \ar[r] \ar@{^{(}->}[d] &E_{i+1}^{\oplus d} \ar[r] \ar@{^{(}->}[d] &F \ar[r] \ar[d] &0\\
 0 \ar[r] & E_i^{\oplus a} \ar[r] \ar@{->>}[d] &E_{i+1}^{\oplus b} \ar[r] \ar@{->>}[d] &K \ar[r] \ar@{->>}[d] &0\\
 \mathcal{H}^{-1}(L) \ar@{^{(}->}[r] &E_i^{\oplus a-c} \ar[r] &E_{i+1}^{\oplus b-d} \ar[r] &\mathcal{H}^0(L) \ar[r] &0
}}
We have that $\mu(\mathcal{H}^{-1}(L))\leq \mu(E_i)$ and $\mu(F)\geq \mu(E_{i+1})>\mu(E_i)$, the latter inequality is a consequence of Remark \ref{Mu-Inequality}. Hence the condition $\mu(F)<\mu(K)$ is equivalent to $(a\cdot d - c \cdot d)(\ch_0(E_i)\ch_1(E_{i+1})-\ch_1(E_i)\ch_0(E_{i+1}))>0$, that is, if the determinant condition is satisfied.
\end{proof}

\begin{rmk}Analogously, one can prove the same result for sheaves that are a kernel of a surjective morphism $E_{i}^{\oplus a} \rightarrow E_{i+1}^{\oplus b}$. We will use both versions of the previous proposition in the following sections.
\end{rmk}

Next, we find a numerical criterion to find regions where every $3$-step complex of dimension $[0,a,b,c]$ is in $\mathcal{A}^{\beta,\alpha}$. We will need this to prove the uniqueness of the $\lambda$-walls for the instantons. 

\begin{lemma}\label{Lemma-Existence}
 If $L\in \langle S(\mathcal{E) \rangle}$ is an object with $\dim_\mathcal{E}(L)=[0,a_1,a_2,a_3]$, then $L \in \mathcal{A}^{\beta,\alpha}$ when $(\beta,\alpha) \in \bar{R}_{[0,a_1,a_2,a_3]}\cap R_\mathcal{E}$, where \[\bar{R}_{[0,a_1,a_2,a_3]}=\{(\beta,\alpha)\in\mathbb{H}| E_i[3-i]\in \mathcal{A}^{\beta,\alpha}\text{ for every $i$ with $a_i\neq 0$}\}.\]
\end{lemma}

\begin{proof}
This Lemma is a consequence of $\mathcal{A}^{\beta,\alpha}$ being closed under extensions and that we can decompose $L$ using the truncation into the following, possibly trivial, exact sequences:
\[0 \rightarrow E_3^{\oplus a_3} \rightarrow L \rightarrow \tau_{\leq -1}L \rightarrow 0\] \[0 \rightarrow E_2^{\oplus a_2}[1] \rightarrow \tau_{\leq -1}L \rightarrow \tau_{\leq -2}L \rightarrow 0\] \[0 \rightarrow E_1^{\oplus a_1}[2] \rightarrow \tau_{\leq -2}L \rightarrow E_0^{\oplus a_0} \rightarrow 0.\]
\end{proof}

When the dimension vector is $[a_0,a_1,a_2,0]$ with $a_0 \neq 0$, we need to consider a slight modification of the previous lemma to take into account when the exceptional objects are in $\mathcal{A}^{\beta,\alpha}[1]$. The proof is essentially the same and therefore we will not include one.

\begin{lemma}\label{Lemma-Existence-2}
 If $L\in \langle S(\mathcal{E) \rangle}$ is an object with $dim_\mathcal{E}(L)=[a_0,a_1,a_2,0]$, then $L \in \mathcal{A}^{\beta,\alpha}$ when $(\beta,\alpha) \in \bar{R}_{[a_0,a_1,a_2,0]}\cap R_\mathcal{E}$, where \[\bar{R}_{[a_0,a_1,a_2,0]}=\{(\beta,\alpha)\in \mathbb{H}| E_i[3-i-1]\in \mathcal{A}^{\beta,\alpha}\text{ for every $i$ with $a_i\neq 0$}\}.\]
\end{lemma}

\begin{rmk}
For a given strong exceptional collection $\mathcal{E}$ with quiver region $R_\mathcal{E}$, denote the region $\bar{R}_{[0,a,b,c]}\cap R_\mathcal{E}$ obtained by Lemma \ref{Lemma-Existence} with $a,b,c \neq 0$, by $\bar{R}_\mathcal{E}$. The intersection of the canonical walls $\Upsilon_i$ with $\bar{R}_\mathcal{E}$ will be denoted by $\bar{\Upsilon}_i$. When clear from the context we may suppress the subscript $\mathcal{E}$.
\end{rmk}

\begin{example}
One example we will use throughout the paper is that, for $X=\mathbb{P}^3$ or $X=Q_3$, any object with dimension $[0,a,b,c]$ with respect to the exceptional collections in Example \ref{Def-Quadrica} and \ref{Def-R1} is in $\mathcal{A}^{\beta,\alpha}$ for $(\beta,\alpha)$ in the intersection of the quiver regions obtained in the aforementioned examples and the region displayed in Figure \ref{Image- Exemplo}.

\begin{figure}[htp]
    \centering
    \includegraphics[width=10cm]{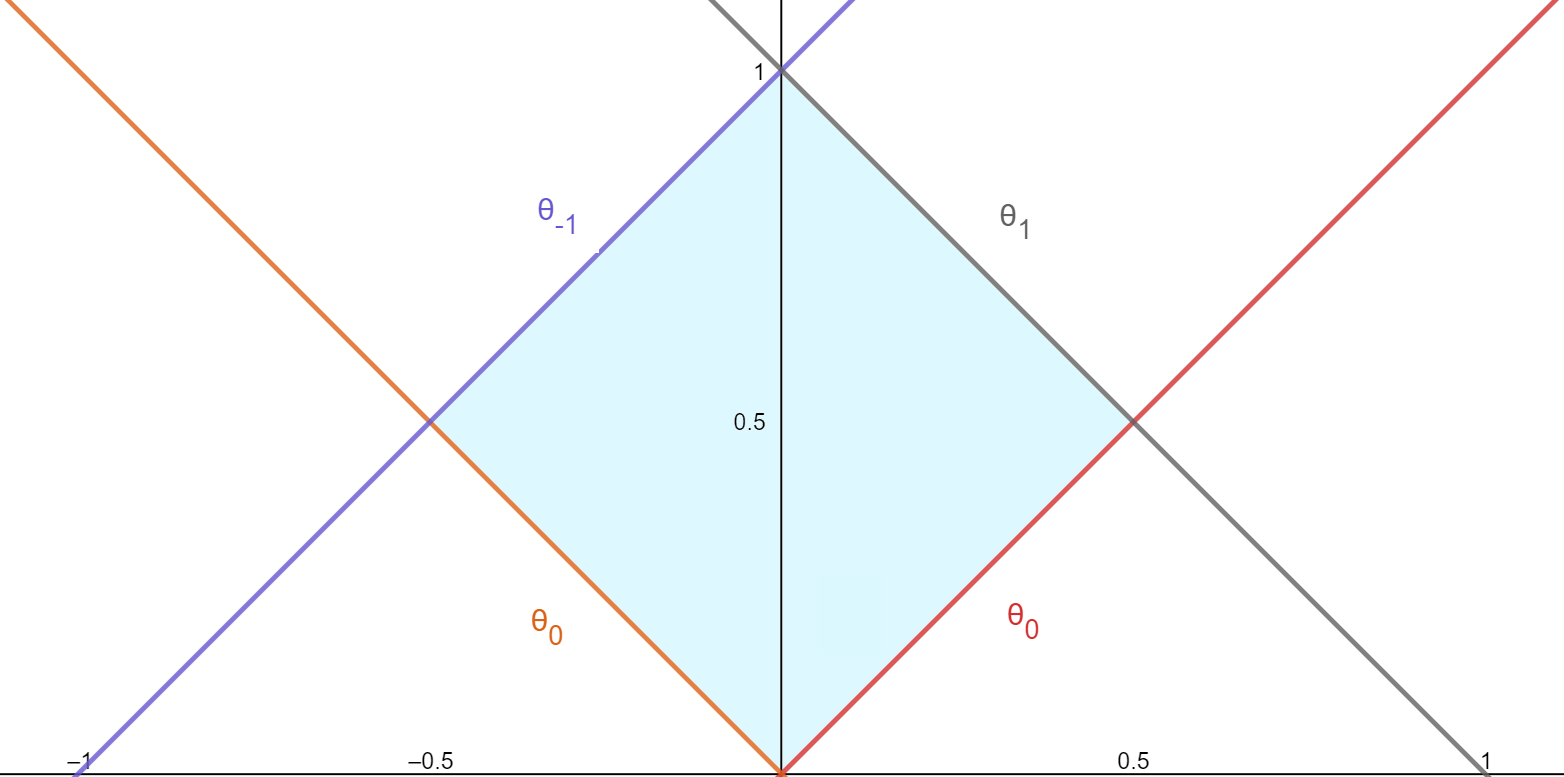}
    \caption{Region $\bar{R}_{[0,a_1,a_2,a_3]}$ obtained by Lemma \ref{Lemma-Existence}, for $\mathcal{E}=\{E_0,\mathcal{O}(-1),\mathcal{O},\mathcal{O}(1)\}$ a full strong exceptional collection in $X=\mathbb{P}^3$ or $X=Q_3$ with $E_0=\mathcal{O}(-2)$ or $E_0=S^{\ast}(-2)$, respectively .}\label{Image- Exemplo}.
\end{figure}

\end{example}

\begin{rmk}\label{Vanishing Wall}
Given an object $F$ in $\langle S(\mathcal{E}) \rangle$  with $\dim(F)=[0,a,b,c]$, the $\lambda$-walls obtained by the truncation functors $\tau_{\leq -1}$ and $\tau_{>-2}$ are vanishing walls because every $3$-step complex can be decomposed in these natural exact sequences \eqref{Truncation Functor}.
\end{rmk}

In the case of concrete examples of linear complexes it is possible to calculate all of its linear subcomplexes. We start by an observation of the well established natural isomorphism for $F,E \in \Db$ and $V,W$ vector spaces \[ \Hom(V \otimes F, W \otimes E)= \Hom(V,W) \otimes \Hom(F,E). \] So that when we fix a base $\{\gamma_0,...,\gamma_k \}$ for the vector space $\Hom(F,E)$ we can describe $\phi \in \Hom(V \otimes F, W \otimes E)$ as  \[ \phi= \Sigma_{i} \gamma_i \phi_i, \] where $\phi_i \in \Hom(V,W)$ are linear transformations. Define $J_\phi^I=\bigoplus_{l=0}^{k}Im(\phi_l|_I)$, where $I$ is a subspace of $V$ and $Im(\phi_l|_I)$ is image of the linear transformation $\phi_l$ restricted to the subspace $I$.

\begin{proposition}\label{Prop-Linear}
Let $\mathcal{E}=\{E_0,...,E_3\}$ be a strong exceptional collection of  sheaves, $S(\mathcal{E})$ its shift. Let \[ K \simeq (V \otimes E_i \overset{T}{\rightarrow} W \otimes E_{i+1})[i]\] be a $2$-step complex in $\langle S(\mathcal{E}) \rangle$. For any subspace $I \hookrightarrow V$, the subcomplexes of $K$ of the form $I \otimes E_i \overset{S}{\rightarrow} J \otimes E_{i+1}$ satisfies $J_T^I \subset J$. Furthermore, $I \otimes E_i \overset{{T|_I}}{\rightarrow} J_T^I\otimes E_j$ is a subobject of $K$ in $\Shi$. 
\end{proposition}

\begin{proof}
As previously observed, given $K \simeq (V \otimes E_i \overset{T}{\rightarrow} W \otimes E_{i+1})$ and a base $\{\gamma_0,...,\gamma_l\}$ for $\Hom(E_i,E_{i+1})$ then $T$ can be factored into $\Sigma_j T_j \otimes \gamma_j$ so that composing with $I\otimes E_i \hookrightarrow V \otimes E_i$, the morphism induced by the inclusion $I \hookrightarrow V$, we obtain a natural commutative diagram\\
\centerline{\xymatrix{
I\otimes E_i \ar[r]^{T|_I} \ar@{^{(}->}[d] &J_T^I\otimes E_{i+1} \ar@{^{(}->}[d]\\
V \otimes E_i \ar[r]^T &W \otimes E_{i+1}.
}}
The map $T|_I$ has a well-defined image in $J_T^I\otimes E_{i+1}$ due to the property that every $T_l|_I$ has its image inside $J_T^I$. The same argument proves that every sub-object of $K$ with $I\otimes E_i$ as its $E_i$ coordinate has to factorize.
\end{proof}

\section{Instanton Sheaves}\label{Sec-Ins}

We are ready to apply the quiver regions to study the stability of the \emph{instanton sheaves} of rank $2$ and rank $0$ over some Fano threefolds. These are really important objects inside the moduli space of Gieseker-semistable sheaves, being linear sheaves and for their nice cohomological properties. Another reason to apply the methods described in the previous section is that instantons can be defined as cohomology of linear monads of exceptional objects, see \cite{JMT,Faen,Kuz,Cos}. 

For a given smooth projective Fano threefold $X$ of Picard rank one and index $i_X=2q_X+r_X$, with $r_X$ the remainder of the division of $i_X$ by $2$, an instanton bundle of charge $c$ is a rank $2$ vector bundle $E$ with $c_2(E)=c L_X$ satisfying the conditions:
\begin{center}
    $E(-q_X) \simeq E^\ast(q_X) \otimes \omega_X$ and $H^1(X,E(-q_X))=0$.
\end{center}

This is the definition given in \cite{Faen}, it agrees with the definition given in \cite{Kuz} and with the usual notion of instanton bundle on $\mathbb{P}^3$. It does not agree with the instanton as proposed in \cite{Cos}, as the instanton obtained in \cite{Faen} for the smooth quadric is an odd instanton ($c_1=-1$) and the one obtained in \cite{Cos} is an even instanton ($c_1=0$). To define the instanton sheaf $E$ in $\mathbb{P}^3$ we use \cite{JMT} definition of a torsion free sheaf satisfying $h^0(E(-1))=h^1(E(-2))=h^2(E(-2))=h^3(E(-3))=0$. 

All of these objects and definitions can be realized as cohomology of a monad defined by exceptional objects. Recall that a monad is a complex of coherent sheaves \[ 0 \rightarrow A \rightarrow B \rightarrow C \rightarrow 0\] such that $A\rightarrow B$ and $B\rightarrow C$ is a monomorphism and an epimorphism, respectively. 

Let $X$ be either $Q_3$ or $\mathbb{P}^3$, the even rank $2$ instantons (bundles in $Q_3$ and sheaves in $\mathbb{P}^3$) are defined as cohomology of a monad of the form 

\begin{equation}\label{Instaton-Defini-Monad}
    0 \rightarrow \mathcal{O}_X(-1)^{\oplus c} \overset{f}{\rightarrow} \mathcal{O}_X^{\oplus 2c+2} \overset{g}{\rightarrow} \mathcal{O}_X(1)^{\oplus c} \rightarrow 0.
\end{equation}

The positive integer $c$ in the monad is the called charge of the instanton $I=\ker(g)/\im(f)$. For the case $X=\mathbb{P}^3$ the sheaf describing the singularities of an instanton sheaf is also a kind of instanton, so-called the rank $0$ instanton $Q$, and it is defined by the exact resolution 

\begin{equation}
    0 \rightarrow \mathcal{O}(-1)^{\oplus d} \rightarrow \mathcal{O}^{\oplus 2d} \rightarrow \mathcal{O}(1)^{\oplus d} \rightarrow Q \rightarrow 0,
\end{equation}

as can be seen in \cite[Proposition 3]{JMT}. The Chern characters $\ch(I)=(2,0,-c,0)$ and $\ch(Q)=(0,0,d,0)$ are associated to the rank $2$ and $0$ instanton sheaves, respectively. 

Now we fix the notation that is going to be used in this section.

\begin{itemize}
    \item Fix $X$ to be equal to $\mathbb{P}^3$ or $Q_3$.
    \item Let $\Shione$ represent either $\Shial$ or $\Shitwo$, as defined in Examples \ref{Def-R1} and \ref{Def-Quadrica}, this is because we are mainly interested in the objects of these exceptional collections which are $\mathcal{O}_X(i)$ with $i=-1,0,1$.
    \item Let us denote by $\bar{R}$ the region obtained by Lemma \ref{Lemma-Existence} with respect to $\Shione$ and $\bar{R}_+,\bar{R}_-$ the points in $\bar{R}$ with $\beta>0$ and $\beta\leq0$, respectively.
    \item Fix $s=1/3$ for the rest of the paper and suppress the subscript $s$ from the definitions of walls, stability function and etc. Also, fix integers $k\geq 1$ and $c \geq 1$.
    \item The walls $\Upsilon_1,\Upsilon_2,\Upsilon_3$ will be defined for the object $E$ with $\dim(E)=[0,c,2c+k,c]$, as in Remark \ref{rmk-lambwall}. Their subset that can be actual $\lambda$-walls will be denoted by $\bar{\Upsilon}_1,\bar{\Upsilon}_2,\bar{\Upsilon}_3$. These are respectively $\Upsilon_1 \cap \bar{R}_-$, $\Upsilon_2 \cap \{(0,\alpha) \in R_1| \alpha^2\geq \frac{1}{3}\}$, $\Upsilon_3 \cap \bar{R}_+$.
\end{itemize}

We start by proving a complement of \cite[Proposition 2.8]{AO} to apply also to instanton sheaves, instead of only instanton bundles. When $X=Q_3$, the analogous of \cite[Proposition 2.8]{AO} was proved in \cite[Proposition 3.3]{Cos}. Combining these results with Proposition \ref{Prop-Mustability} will be responsible for the existence of the actual $\lambda$-walls for the instantons.

\begin{lemma}\label{Lemma-28-referencia}
Let  \begin{equation}\label{Monada-Refencia28}
    0 \rightarrow \mathcal{O}_X(-1)^{\oplus b} \rightarrow \mathcal{O}_X^{\oplus 2b+2} \rightarrow \mathcal{O}_X(1)^{\oplus b} \rightarrow 0 
\end{equation}
be a monad over $X$ with middle cohomology a torsion free sheaf $E$. Then we can decompose \eqref{Monada-Refencia28} in two ways:
\begin{equation}\label{Eq-1-Truncamento}
    0 \rightarrow \mathcal{O}_X(-1)^{\oplus b} \rightarrow \mathcal{O}_X^{\oplus 2b+2} \rightarrow K_1 \rightarrow 0
\end{equation}
\begin{equation}\label{Eq-1-Truncamento-2}
     0 \rightarrow E \rightarrow K_1 \rightarrow \mathcal{O}_X(1)^{\oplus b} \rightarrow 0
\end{equation}
and also,
\begin{equation}\label{Eq-2-Truncamento}
    0 \rightarrow K_2 \rightarrow \mathcal{O}_X^{\oplus 2b+2} \rightarrow \mathcal{O}_X(1)^{\oplus b} \rightarrow 0
    \end{equation}
    \begin{equation}\label{Eq-2-Truncamento-2}
    0 \rightarrow \mathcal{O}_X(-1)^{\oplus b} \rightarrow K_2 \rightarrow E \rightarrow 0
\end{equation}

Such that $K_2$ is always a $\mu$-stable bundle and $K_1$ is a $\mu$-stable bundle if $E$ is also a vector bundle. 
\end{lemma}

\begin{proof}
The case where $E$ is a vector bundle is a direct application of \cite[Proposition 2.8]{AO} and \cite[Proposition 3.3]{Cos} to the exact sequence \eqref{Eq-1-Truncamento} and \eqref{Eq-1-Truncamento-2}, and by dualizing \eqref{Eq-2-Truncamento} and \eqref{Eq-2-Truncamento-2}. For the case where $E$ is not locally free, notice that $K_2$ is a locally free sheaf by being a kernel of a map between vector bundles, and one can dualize sequence \eqref{Eq-2-Truncamento-2} to obtain
\begin{equation}\label{Eq-DualizedSeq}
    0 \rightarrow E^\ast \rightarrow K^\ast_2 \rightarrow \mathcal{O}_X(1)^{\oplus b} \rightarrow \mathcal{E}xt^1(E,\mathcal{O}_X) \rightarrow 0.
\end{equation} 

Let $L$ be the kernel of $\mathcal{O}_X(1)^{\oplus b} \rightarrow \mathcal{E}xt^1(E,\mathcal{O}_X)$ and $S=\Supp(\mathcal{E}xt^1(E,\mathcal{O}_X))$, $S$ is also the singular set of the instanton sheaf $E$. We know from \cite[Proposition 1.1.10]{HL} that $\dim(S)\leq 1$ because $E$ is torsion free and therefore we can consider the sequence \eqref{Eq-DualizedSeq} over the open subset $U=X\setminus S$. 

In this situation, we obtain that $L=\mathcal{O}_X(1)^{\oplus b}|_U$, hence we are in place to apply the same argument as in \cite[Proposition 2.8]{AO} with the caveat that $K_2$ being locally free implies that $\wedge^qK_2(l)$ is also normal for any $q$ and $l$, in the sense of \cite[Definition 1.1.11]{OS}, making it so that $h^0(\wedge^qK_2(l)|_U)=h^0(\wedge^qK_2(l))$.
\end{proof}

Using this technique we are not able to prove that $K_1$ is also $\mu$-stable in the non-locally free case because there is no way to establish that $K_1$ is a vector bundle, being a cokernel of a map between vector bundle and an extension of a torsion-free sheaf and a vector bundle. Furthermore, if the cohomology of the monad is a non-locally free sheaf then $K_1$ will never be $\mu$-stable, this will be clear in the following results.

\begin{lemma}\label{Lemma-DualMonad}
The object $\tau_{>-2}(F)$ satisfies the (semi)determinant condition if and only if the $\tau_{\leq -1}(F^\vee)$ also satisfies the (semi)determinant condition, for any $F\in \Shione$ with $\dim(F)=[0,a,b,c]$.
\end{lemma}

\begin{proof}
This is a direct consequence of the exactness and involution property of the derived dual. That is, exact triangles are kept exact after applying the derived dual $(-)^\vee: \Db \rightarrow \Db$ and that $(-)^{\vee \vee}=\text{Id}_{\Db}$. The fact that the exceptional objects $E_i$ are locally free and $E_i^\vee=E_{4-i}$ establishes that the dual linear complex is still in $\Shione$.
\end{proof}

\begin{proposition}\label{Prop-T1-T3}
Let $(\beta,\alpha) \in R_1\cap\{(0,\alpha) \in \mathbb{H}| \alpha^2\geq \frac{1}{3}\}$. Then any object $E$, with $\dim(E)=[0,c,2c+k,c]$,  $\lambda_{\beta,\alpha}$-semistable at both $\bar{\Upsilon}_1$ and $\bar{\Upsilon}_3$ is $\lba$-stable, unless there exists an object $F \in \Shione$ such that $\lba(F)=\lba(E)$ for every $(\beta,\alpha) \in \bar{R}$.
\end{proposition}

\begin{proof}
Let $E$ be an object with $\dim(E)=[0,c,2c+k,c]$ and $F$ a destabilizing $\lambda_{\beta,\alpha}$-semistable subobject for $E$ at $\alpha^2>1/3$ . For any point $(0,\alpha)\in R_1$, since $E,F \in \mathcal{A}^{0,\alpha}$ with $\lambda_{0,\alpha}(F)\geq\lambda_{0,\alpha}(E)$ for every $\alpha$ in then open interval $(1/\sqrt{3},1/\sqrt{3}+\epsilon)$ for some small $\epsilon>0$, then $F \in \Shione$ and $\dim(F)=[0,f,g,h]$ for some $f,g,h\in \mathbb{Z}_{\geq 0}$. We can determine the following commutative diagram

\begin{equation}\label{Trucantion-Exact-Seq-1}
\begin{gathered}
    \xymatrix{
& 0 \ar[d] &0 \ar[d] &0 \ar[d]\\
 0 \ar[r] &\tau_{>-2}(F) \ar[d] \ar[r]  &F \ar[r] \ar[d]  &\mathcal{O}_X(-1)^{\oplus f} \ar[r] \ar[d]  &0\\
 0 \ar[r] &\tau_{>-2}(E) \ar[r] \ar[d] &E \ar[r] \ar[d]  &\mathcal{O}_X(-1)^{\oplus c} \ar[r] \ar[d] &0\\
 0 \ar[r] &\tau_{>-2}(Q) \ar[r] \ar[d] &Q \ar[r] \ar[d] &\mathcal{O}_X(-1)^{\oplus c-f} \ar[r] \ar[d]  &0\\
 &0 &0 &0
}
\end{gathered}
\end{equation}
where $Q$ is the quotient of $F \hookrightarrow E$ in $\Shione$. Since the $\lambda$-wall $\bar{\Upsilon}_1$ is actual, we can see that $\tau_{>-2}(E)$ is $\lba$-semistable and by Lemma \ref{Lemma-Determinant Condition} we know that $\tau_{>-2}(E)$ satisfies the determinant condition. Now, applying said condition to the exact sequence $\tau_{>-2}(F) \hookrightarrow \tau_{>-2}(E)$ we obtain the inequality \begin{equation}\label{ineq-1}
    (c+k)h\geq c(g-h).
\end{equation} 
Similarly, we can do the same argument for the stupid truncation $\tau_{\leq -1}$ applied to the exact sequence  $F \hookrightarrow E \twoheadrightarrow Q$ to obtain the inequality 
\begin{equation}\label{ineq-2}
    (c+k)f\leq c(g-f).
\end{equation}
Considering Lemma \ref{Lemma-Determinant-Calculation}, we see that if both of these inequalities were an equality then $f_{E,F}\equiv 0$. If that is not the case then one of those inequalities is a strict inequality

Now, analyze the $\lambda$-slope of $F$ and $E$ at the line $\beta=0$ using that $\lambda_{0,\alpha}(E)=\tau_{0,\alpha}(E)=0$ for all $\alpha>0$ so that if $F$ destabilizes $E$ at $\tilde{\alpha} \in (1/\sqrt{3},1/\sqrt{3}+\epsilon)$ for some small $\epsilon>0$ then $\tau_{0,\tilde{\alpha}}(F)> 0$. Applying the linearity of the operator $\tau_{0,\tilde{\alpha}}$ over the Chern characters we arrive at \begin{equation}\label{Equation-Vertical}
    \begin{split}
    \tau_{0,\tilde{\alpha}}(F) &=f\tau_{0,\tilde{\alpha}}(\mathcal{O}_X(-1))-g\tau_{0,\tilde{\alpha}}(\mathcal{O}_X)+h\tau_{0,\tilde{\alpha}}(\mathcal{O}_X(1))\\
    &=-f(\frac{1}{6}-\frac{\tilde{\alpha}^2}{2})+h(\frac{1}{6}-\frac{\tilde{\alpha}^2}{2})=(h-f)(\frac{1}{6}-\frac{\tilde{\alpha}^2}{2})
    \end{split}
\end{equation}
   
which by our choice of $\tilde{\alpha}$ implies $f\geq h$.

Combining the inequalities \eqref{ineq-1}, \eqref{ineq-2} and $f\geq h$ it is clear that we arrived at a contradiction.
\end{proof}

\begin{theorem}\label{Theorem-Uniqueness} Let $E \in \Shione$ be an object with $\dim(E)=[0,c,2c+k,c]$ and suppose it has an actual $\lambda$-wall over any two of the three canonical walls $\bar{\Upsilon}_i$, $i=1,2,3$. Then these two walls are the only actual walls for the object $E$ in $R_1$. Moreover, $E$ is $\lba$-semistable outside the respective actual walls, being strictly semistable if and only if there exists an subobject with dimension vector a multiple of $[0,c,2c+k,c]$.
\end{theorem}

We will divide the proof in two situations: Either the walls to be considered are $\bar{\Upsilon}_1$ and $\bar{\Upsilon}_3$ or one of them is $\bar{\Upsilon}_2$. The former case is a combination of the latter ones and Proposition \ref{Prop-T1-T3}. The proof of the case with both $\bar{\Upsilon}_1$ and $\bar{\Upsilon}_2$ being actual is the same argument as the case as in the with $\bar{\Upsilon}_2$ and $\bar{\Upsilon}_3$ being actual, so that we will only show one.

\begin{proof}
\textit{Case ($\bar{\Upsilon}_1,\bar{\Upsilon}_2$):} Assume that $F$ is an object with $\dim(F)=[0,f,g,h]$ that determines an actual wall for $E$. Because $E$ is $\lba$-semistable at $\bar{\Upsilon}_2$, we know that $h\geq f$ by \eqref{Equation-Vertical} . By the functoriality of the truncation functor we know that $\tau_{>-2}(F)$ is a subobject of $\tau_{>-2}(E)$, and by Lemma \ref{Lemma-Determinant Condition} we have with $\det(A):=h(2c+k)-cg \geq 0$. Now we can apply Lemma \ref{Lemma-Determinant-Calculation} to determine the equation for the $\lambda$-wall determined by $F$ as

\begin{equation}
f_\sigma(E,F) =\begin{vmatrix}
f_\sigma(\mathcal{O}_X(1),\mathcal{O}_X) &f_\sigma(\mathcal{O}_X(1),\mathcal{O}_X(-1)) &f_\sigma(\mathcal{O}_X,\mathcal{O}_X(-1))\\
c &2c+k &c\\
f &g &h
\end{vmatrix}.
\end{equation}

The fact that its given by a determinant is useful to reduce the wall equation to

\begin{equation}
f_\sigma(E,F) =\begin{vmatrix}
f_\sigma(\mathcal{O}_X(1),\mathcal{O}_X) &f_\sigma(\mathcal{O}_X(1),\mathcal{O}_X(-1)) &f_\sigma(\mathcal{O}_X,\mathcal{O}_X(-1))\\
c &2c+k &c\\
f-h &\frac{-\det(A)}{c} &0
\end{vmatrix}.
\end{equation}
So that we can describe $f_\sigma(E,F)$ as \begin{equation}\label{Wall Reduction -1}
     f_\sigma(E,F)=\frac{\det(A)}{6}(-\beta^3)+(f-h)f_\sigma(E,\mathcal{O}_X(-1)[2]).
\end{equation}
Moreover, if $\sigma \in \bar{R}_-$ is outside the numerical $\lambda$-wall defined by $\mathcal{O}_X(-1)[2]$ then \linebreak $f_\sigma(E,\mathcal{O}_X(-1)[2])$ is strictly negative, making $f_\sigma(E,F)=0$ if and only if $F$ defines either the $\lambda$-wall $\Upsilon_2$, $\Upsilon_{3}$ or $\dim(F)$ is proportional to $\dim(E)$. To conclude we just have to prove that it does not have any $\lambda$-wall in $\bar{R}_+$, but that is just a consequence of it having the wall $\bar{\Upsilon}_2$ as actual $\lambda$-wall and from observing that the outside of this wall is exactly $\beta<0$, making it that any object defining the wall $\bar{\Upsilon}_2$ should destabilize $E$ for $\beta>0$.

\textit{Case ($\bar{\Upsilon}_1,\bar{\Upsilon}_3$):} For that, we just need to apply Proposition \ref{Prop-T1-T3} because in this case we have that $E$ will $\lba$-semistable at $\bar{\Upsilon_2}$ and therefore, from cases $i=1,2$ and $i=2,3$, we can conclude that $E$ does not have any other wall beyond $\bar{\Upsilon}_1$ and $\bar{\Upsilon}_3$
\end{proof}

\begin{rmk}
The statement of Theorem \ref{Theorem-Uniqueness} can be regarded in another, less geometric, manner, where instead of demanding that the canonical walls are actual we assumed the respective truncation functor applied to $E$ satisfies the determinant condition. For that we would just need to adjust the notion of the determinant condition to the case of $\bar{\Upsilon}_2$ to which we would say $E$ satisfies the middle determinant condition if every for every subobject $F$ with $\dim(F)=[0,f,g,h]$ it satisfied $f>h$.
\end{rmk}

As a direct consequence of Theorem \ref{Theorem-Uniqueness}, locally free instanton sheaves shifted by $1$ are $\lba$-stable at every point of $R_1$, outside the walls $\bar{\Upsilon}_1$ and $\bar{\Upsilon}_3$.

\begin{corollary}\label{cor-inststab}
For any locally free rank $2$ instanton $I$ with charge $c$, $I[1]$ is $\lba$-stable for every $(\beta,\alpha) \in \bar{R}$ outside of both $\lambda$-walls $\bar{\Upsilon}_1$ and $\bar{\Upsilon}_3$.
\end{corollary}

The description of any wall by equation \eqref{Wall Reduction -1} gives us a corollary about the intersection of the numerical $\lambda$-walls with $\bar{\Upsilon}_i$.

\begin{corollary}\label{Corollary-uniqueness}
Let $E$ be an object with $\dim(E)=[0,c,2c+k,c]$. If either $\tau_{> -2}E$ or $\tau_{\leq -1}E$ satisfy the determinant condition, then no other numerical $\lambda$-wall can cross the respective actual $\lambda$-wall, determined by the truncation, at a point different then $(0,1/\sqrt{3})$, unless this numerical $\lambda$-wall determines the same $\lambda$-wall as the respective $\bar{\Upsilon}_i$.
\end{corollary}

We will prove this result only for $i=1$, as for the other case of $i=2$ is analogous.

\begin{proof}
First, observe that $\bar{\Upsilon}_1$ is only an actual $\lambda$-wall for the points $(\beta,\alpha)$ where $\beta\leq 0$. The point $\beta=0$ in $\bar{\Upsilon}_1$ has $\alpha^2=1/3$. Now, we assume there exists a stability condition $\sigma$ at the intersection $\bar{\Upsilon}_1 \cap \Upsilon_{E,F}$ so that $f_\sigma(E,F)=0$, but we know from $\sigma \in \bar{\Upsilon}_1$ that $f_\sigma(E,\mathcal{O}_X(-1)[2])=0$ and $\tau_{>-2}(E)$ satisfies the determinant condition. Moreover, using equation \eqref{Wall Reduction -1} and its notation, we also can conclude that $\det(A)=0$. This implies that either $\bar{\Upsilon}_1 = \Upsilon_{E,F}$, without taking in to account the orientation, or $\dim(F)$ is proportional to $\dim(E)$ and in that case $E$ has the same $\lambda$-slope as $F$ at every point.
\end{proof}

Next, we analyze the stability of the rank $0$ instanton sheaves. This will be important in determining the uniqueness of the walls for the non-locally free instantons sheaf. 

\begin{lemma}\label{lemma-rank0}
On the case of $X=\mathbb{P}^3$, rank $0$ instanton sheaves are $\lambda_{\beta,\alpha}$-semistable for every point $(\beta,\alpha)\in \bar{R}\cap\{(0,\alpha) \in \mathbb{H}| \alpha^2\geq \frac{1}{3}\}$.
\end{lemma}

\begin{proof}
From Lemma \ref{Lemma-Existence} we conclude that $Q \in \mathcal{A}^{0,\frac{1}{\sqrt{3}}}$ and that $Q$ is an extension of $\mathcal{O}_X(-1)^{\oplus d}[2]$, $\mathcal{O}_X^{\oplus 2d}[1]$ and $\mathcal{O}_X(1)^{\oplus d}$ in $\Shial$ but all of these objects have the same $\lambda$-slope in $(0,1/\sqrt{3})\in \mathbb{H}$. Therefore $Q$ is $\lambda_{0,\frac{1}{\sqrt{3}}}$-semistable at that point. 

Suppose that $F$ is a destabilizing object for $Q$ at $\alpha \in (1/\sqrt{3},1/\sqrt{3}+\epsilon)$ for some $\epsilon>0$. Then $F$ is $\lambda_{0,\alpha}$-semistable and $F \in \Shial$ with $\dim(F)=[0,f,g,h]$. As in equation \eqref{Equation-Vertical}, we know that $\lambda_{0,\alpha}(Q)=0$ for every $\alpha$ and if $F$ destabilizes $Q$ then $f>h$.

But now we can study the following exact sequence in $\mathcal{A}^{0,\frac{1}{\sqrt3}}$
\begin{equation}
    0 \rightarrow F \rightarrow Q \rightarrow G \rightarrow 0.
\end{equation}

Applying the cohomology functor $\HH^{0}_{\Coh}$ and $\HH^{0}_{\mathcal{B}^\beta}$ for $\beta=0$ to this sequence we obtain that $\HH^{-2}(F)=0$, $\HH^{-2}(G)=\HH^{-1}(F)$, $\HH^{-1}_{\mathcal{B}^0}(F)=0$. But $\ch_1^{0}(F)=\ch_1^0(\mathcal{H}_{\mathcal{B}^0}^{0}(F))=h-f$ should be positive because $F \in \mathcal{B}^{0}$, which is impossible due to $F$ destabilizing $Q$. 
\end{proof}

\begin{corollary}\label{Corollary-Unicity}
At the intersection, with $(\beta,\alpha)\in \bar{R}\cap\{(0,\alpha) \in \mathbb{H}| \alpha^2\geq \frac{1}{3}\}$, instanton sheaves are $\lambda_{\beta,\alpha}$-semistable.
\end{corollary}

\begin{proof}
Notice that the double dual of an instanton sheaf is a locally free instanton sheaf, see \cite{JMT}, and the existence of the shifted double dual sequence \[ 0 \rightarrow Q \rightarrow I[1] \rightarrow I^{\ast\ast}[1] \rightarrow 0\] in $\mathcal{A}^{0,\alpha}$.
\end{proof}

\begin{rmk}
If an analogous result as \cite[Proposition 2.8]{AO} were found to be true, it would be possible to generalize Corollary \ref{Corollary-Unicity}, Lemma \ref{lemma-rank0} and Corollary \ref{cor-inststab} to higher rank instanton sheaves.
\end{rmk}

\begin{example}
\begin{figure}[htp]
    \centering
    \includegraphics[width=10cm]{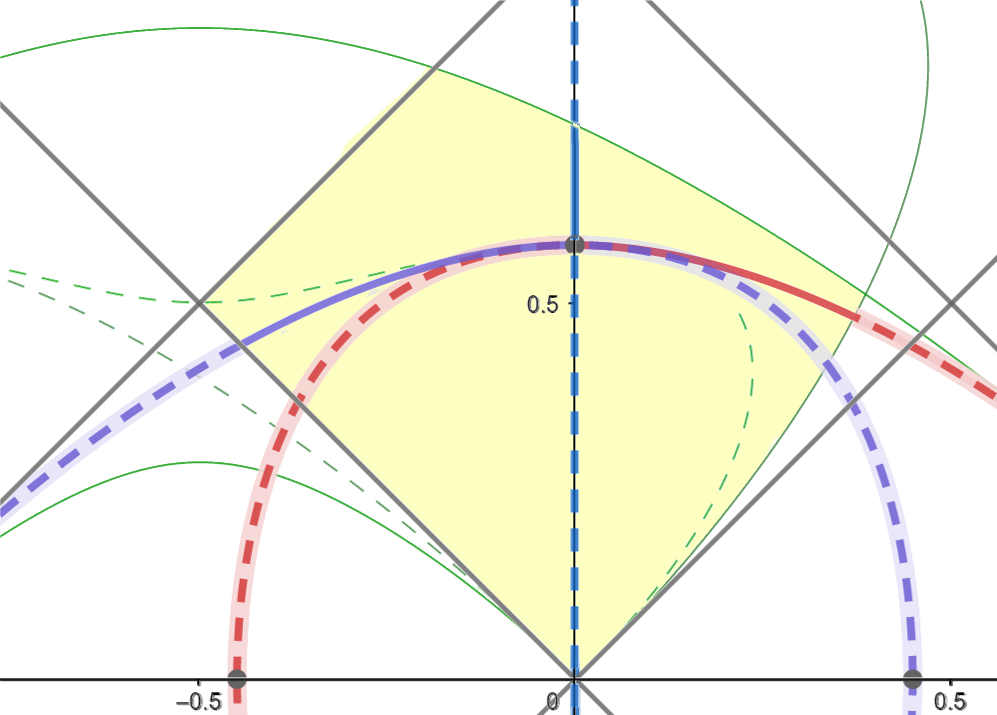}
    \caption{For $X=\mathbb{P}^3$: the yellow region represents $\bar{R}$, the purple curve is $\bar{\Upsilon}_1$, the blue curve is $\bar{\Upsilon}_2$, the red curve is $\bar{\Upsilon}_3$ and the doted curves represent the non-necessarily actual parts of these walls. The green curves, either doted or not, represent the curves determining the quiver region.}\label{Imagem - Parede Real Instantons}
\end{figure}

Figure \ref{Imagem - Parede Real Instantons} represents the actual $\lambda$-walls discussed in this section for the instantons with charge $c=2$ on $\mathbb{P}^3$. For different values of $c$ the picture does not change much, at least in this restricted region.

\end{example}

We can summarize the results relating the stability of the instanton sheaves in the next proposition.

\begin{proposition}\label{Lemma-Instantons Moduli}
On the case of $X=\mathbb{P}^3$, let \[ \mathcal{I}(c)=\{I\in \Shial|\text{ $\dim(I)=[0,c,2c+2,c]$, $\HH^{-2}(I)=0$ and $\HH^0(I)=0$}\}.\]
Then any object in $\mathcal{I}(c)$ is $\lba$-stable, for every $(\beta,\alpha) \in \bar{R}_-$ outside of $\bar{\Upsilon}_1$.
\end{proposition}

\begin{proof}
Let $c\geq 1$ be an integer and $(\beta,\alpha)\in \bar{R}_-$ outside of $\bar{\Upsilon}_1$. Firstly, observe that the objects in $\mathcal{I}(c)$ are cohomology of the linear monad \eqref{Instaton-Defini-Monad} and so are a shift by $[1]$ of the instanton sheaves. Let $I[1]$ be one of these objects. By a combination of Lemmas \ref{Corollary-Unicity}, \ref{Lemma-28-referencia} and Theorem \ref{Theorem-Uniqueness} we can conclude that $I[1]$ is $\lba$-semistable at the walls $\bar{\Upsilon}_1,\bar{\Upsilon}_2$ and it does not have any other wall in $\bar{R}_-$, unless it is semistable at every point of $\bar{R}_-$.

To see how the later is impossible, assume that there exists a $F\hookrightarrow I[1]$ in $\Shione$ such that $\lba(F)=\lba(I[1])$ for every point in $\bar{R}_-$ and outside of $\bar{\Upsilon}_1$. We know from Lemma \ref{Lemma-Determinant-Calculation} that an object $F$ has the same $\lba$-slope as $I[1]$ for a $2$-dimensional region if and only if its dimension vector is a multiple of $[0,c,2c+2,c]$. But this impossible because $\tau_{>-2}(I[1])$ would have a subobject with dimension vector a multiple of its dimension, which is impossible because $\tau_{>-2}(I[1])$ is $\lba$-stable by Proposition \ref{Prop-Mustability}. 

\end{proof}

The first part of the main theorem finds a description for the Bridgeland (semi)stable objects inside of the region $\bar{R}_-$.

\begin{mthm}\label{Main Theorem - 3}
On the case of $X=\mathbb{P}^3$, for any $(\beta,\alpha) \in \bar{R}_-$ outside and sufficiently close to the $\lambda$-wall $\bar{\Upsilon}_1$, inside the Bridgeland moduli space $\mathcal{M}^{ss}_{\beta,\alpha}[0,c,2c+2,c]$ we have the set \[ \mathcal{N}(c):=\bigcup_{T \in \mathcal{M}^{s}_{\beta,\alpha}[0,0,2c+2,c] }\{F \in \Ext^1(\opn(-1)^{\oplus c}[2],T)|\text{ with $\Hom(\opn(-1)[2],F)=0$}\}. \] 

We have the inclusions $\mathcal{I}(c) \hookrightarrow \mathcal{N}(c) \hookrightarrow \mathcal{M}^{ss}_{\beta,\alpha}[0,c,2c+2,c]$ such that $\mathcal{I}(c)$ is an open subset of $ \mathcal{M}^{ss}_{\beta,\alpha}[0,c,2c+2,c]$ and if the charge $c$ is odd then $\mathcal{N}(c)$ is the whole moduli space $\mathcal{M}^{s}_{\beta,\alpha}[0,c,2c+2,c]=\mathcal{M}^{ss}_{\beta,\alpha}[0,c,2c+2,c]$. 
\end{mthm}

\begin{proof}
Applying Corollary \ref{Corollary-uniqueness} it is clear that there exists a region in $\bar{R}_-$ outside $\bar{\Upsilon}_1$ where no other numerical $\lambda$-wall goes through. Let $W$ be this open neighborhood. Since the wall $\bar{\Upsilon}_1$ is a vanishing wall for the dimension vector $[0,c,2c+2,c]$ then every $\lba$-semistable object for $(\beta,\alpha) \in W$ is of the form  $F \in \Ext^1(\opn(-1)^{\oplus c}[2],T)$ with $T$ satisfying the determinant condition. Now we just have to prove that the necessary and sufficient condition for $F$ to be $\lba$-semistable.

Suppose that $F$ is in $\mathcal{N}(c)$ and is unstable after crossing $\bar{\Upsilon}_1$ to its outside and assume that $L$ is a subobject destabilizing $F$ in $(\beta,\alpha) \in W$ with dimension vector $[0,f,g,h]$. Then by Corollary \ref{Corollary-uniqueness} $L$ determines the same wall as $\bar{\Upsilon}_1$ and by equation \eqref{Wall Reduction -1} we have $\det(A)=h(2c+2)-cg=0$. If $ \tau_{>-2}(F)$ were in $\mathcal{M}^{s}_{\beta,\alpha}[0,0,2c+2,c]$ we would have either $\tau_{>-2}(L)$ with dimension $[0,0,2c+2,c]$ or equal to zero, in the former case $L$ would never destabilize $F$ due to Equation \eqref{Wall Reduction -1} and in the later case $L=\mathcal{O}_X(-1)^{\oplus k}[2]$ for some $k>0$.

Now, if $c$ is odd then every $2$-step complex with dimension $[0,0,2c+2,c]$ satisfies the determinant condition if and only if it is $\lambda_{\beta,\alpha}$-stable after crossing its $\lambda$-wall. Even more, an object $F$ is $\lba$-semistable for $(\beta,\alpha) \in W$ if and only if it is $\lba$-stable. 

The openess of $\mathcal{I}(c)$ comes from the fact that general maps in the linear complex \[ \mathcal{O}_X(-1)^{\oplus c} \overset{f}{\rightarrow} \mathcal{O}_X^{\oplus 2c+2} \overset{g}{\rightarrow} \mathcal{O}_X(1)^{\oplus c}\] determines a monad.
\end{proof}

\begin{rmk}
The condition $\Hom(\opn(-1)[2],F)=0$ for an object $F$ to be $\lba$-stable is a necessary condition for all $(\beta,\alpha)$ outside of $\bar{\Upsilon}_1$. This condition is equivalent to the vector space of global sections of $\mathcal{H}^{-2}(F)(1)$ being zero.
\end{rmk}

The analogous occurs to the right-hand side of $\bar{\Upsilon}_2$, but now we have the locally free and \emph{perverse instanton sheaves} instead of the instanton sheaves. For that we will need the following definition present in \cite[Definition 5.6]{AJM}.

\begin{definition}
An object $E \in \Db$ is a perverse instanton sheaf if it is isomorphic to a linear complex of the form \[ \mathcal{O}_X(-1)^{\oplus c} \overset{f}{\rightarrow} \mathcal{O}_X^{\oplus a} \overset{g}{\rightarrow} \mathcal{O}_X(1)^{\oplus c}\] such that the left derived dual of the restriction to a line, $L j^\ast(E)$, is a shift of a  sheaf, where $j: l \hookrightarrow X$ is a line inside of $X$.
\end{definition}

As noted in \cite{JD}, every derived dual of an instanton sheaf is a perverse instanton sheaf but the converse is not true. 

\begin{example}
We can take, for example, $E\in \Dbal(\mathbb{P}^3)$ to be the linear complex \[ 0 \rightarrow \opn \overset{\begin{pmatrix}
-y \\
x \\
0 \\
w
\end{pmatrix}}{\longrightarrow} \opn^{\oplus 4} \overset{\begin{pmatrix}
x &y &z &0
\end{pmatrix}}{\longrightarrow} \opn \]  which has cohomology equal to $\mathcal{O}_p$ at the highest degree, for the point $p\in \mathbb{P}^3$ satisfying the equation $x=y=z=0$. Therefore, we can choose a line that does not go through the point $p$ in order to conclude that $L j^\ast(E)$ is a sheaf element, making $E$ into a perverse instanton sheaf.

Moreover, it is not the derived dual of any instanton sheaf as its cohomology at the highest degree would necessarily need to be pure of dimension $1$, \cite{JMT}. Even more, we can calculate $E^\vee$ to determine that its cohomology at the highest degree would be $\mathcal{O}_q$, for the point $q$ determined by $x=y=w=0$, and in that way it cannot be the derived dual of an instanton sheaf because the derived dual is an involution.
\end{example}

\begin{proposition}\label{Lemma- Locally free Instantons}
Let $\mathcal{L}(c)\subset \mathcal{I}(c)$ be the subset of shifted locally free instanton sheaves of charge $c$ and $\mathcal{P}(c)=\{E \in \Shione |\text{ $E^\vee \in \mathcal{I}(c)$}\}$ the subset of perverse instanton sheaves that are dual of instanton sheaves of charge $c$. Then, for every $(\beta,\alpha) \in \bar{R}_+$ outside of $\bar{\Upsilon}_3$, every object in $\mathcal{L}(c)$ and $\mathcal{P}(c)$ is $\lba$-stable.
\end{proposition}

\begin{proof}
This statement is a direct consequence of Proposition \ref{Lemma-Instantons Moduli}, Lemma \ref{Lemma-DualMonad} and Lemma \ref{Lemma-28-referencia}.
\end{proof}

\begin{mthm}\label{Main Theorem - 4}
For any $(\beta,\alpha) \in \bar{R}_+$ outside and sufficiently close to the $\lambda$-wall $\bar{\Upsilon}_3$, inside the Bridgeland moduli space $\mathcal{M}^{s}_{\beta,\alpha}[0,c,2c+2,c]$ we have the set \[ \tilde{\mathcal{N}}(c):=\bigcup_{T \in \mathcal{M}^{s}_{\beta,\alpha}[0,0,2c+2,c] }\{F \in \Ext^1(T,\opn(1)^{\oplus c})|\text{ with $\Hom(F,\opn(1))=0$}\}. \] 

We have the inclusions $\mathcal{L}(c), \mathcal{P}(c) \hookrightarrow \tilde{\mathcal{N}}(c) \hookrightarrow \mathcal{M}^{s}_{\beta,\alpha}[0,c,2c+2,c]$. In addition, if the charge $c$ is odd, then $\tilde{\mathcal{N}}(c)$ is the whole moduli space $\mathcal{M}^{s}_{\beta,\alpha}[0,c,2c+2,c]=\mathcal{M}^{ss}_{\beta,\alpha}[0,c,2c+2,c]$. 
\end{mthm}

\begin{proof}
    The proof is analogous to that of Main Theorem \ref{Main Theorem - 3}, where we consider $\bar{R}_+$ and the outside of the actual $\lambda$-wall $\bar{\Upsilon}_3$ instead of $\bar{R}_-$ and the outside of the actual $\lambda$-wall $\bar{\Upsilon}_1$. The main difference here is, in the case of $\mathbb{P}^3$, that we need to consider as well the objects in $\mathcal{P}(c)$ which come from crossing $\bar{\Upsilon_2}$.
\end{proof}
    
\begin{rmk}          
In the case of $X=Q_3$ we can only obtain that the objects in $\mathcal{L}(c)$ are inside the moduli space of stable objects $\mathcal{M}_{\beta,\alpha}(c)$, due to the fact that we do not have a well-established notion of non-locally free instanton sheaves in $Q_3$ to obtain both our version of Main Theorem \ref{Main Theorem - 3} and, consequently, the stability of the objects in $\mathcal{P}(c)$.  
\end{rmk}

The above construction can be made more explicit in the case of $c=1$. The quiver approach to the study of the moduli space of objects with dimension vector $[0,1,4,1]$ was done by Jardim--daSilva, see \cite{JD}. In the next example, we will show how to obtain the same description of the walls by using this different method.

\begin{example}
Let $c=1$ and $E$ an object in $\Shione$ with $\dim(E)=[0,1,4,1]$. Firstly, observe that $\tau_{>-2}(E)$ or $\tau_{\leq -1}(E)$ satisfies the determinant condition if and only if $E$ is $\lba$-semistable at $\bar{\Upsilon}_2$. This is a consequence of Lemma \ref{Lemma-Determinant-Calculation}, by observing that an object destabilizes $E$ at $\bar{\Upsilon}_2$ if and only if it $\lambda$-destabilizes both $\tau_{>-2}(E)$ and $\tau_{\leq -1}(E)$. 

Applying Theorem \ref{Theorem-Uniqueness} we note that the only possible actual $\lambda$-walls for an object with dimension vector $[0,1,4,1]$ are $\bar{\Upsilon}_i$. Hence, we can apply the Theorem \ref{Theorem-Uniqueness} to prove that instanton sheaves are stable at $\bar{R}_-$, outside of $\bar{\Upsilon}_1$, while locally free and perverse instanton sheaves which are duals of instanton sheaves are stable at $\bar{R}_+$, outside of $\bar{\Upsilon}_3$.

To prove that these are the only stable objects in their respective spaces, we just have to show that if an object $E$ is stable at $\bar{R}_-$ outside $\bar{\Upsilon}_1$ then $\HH^{-2}(E)=0$ and $\HH^0(E)=0$.

\textit{$\HH^{-2}(E)=0$:} If $\HH^{-2}(E)=\opn(-1)$ then $\opn(-1)$ is a direct summand of $E$ and therefore $E$ is nowhere stable, outside the wall $\bar{\Upsilon}_1$. Otherwise, suppose $0\neq\HH^{-2}(E) \neq \opn(-1)$ then we would have $Q=\opn(-1)/\HH^{-2}(E)$ as a subsheaf of $\opn^{\oplus 4}$, which is absurd because $Q$ is a torsion sheaf.

\textit{$\HH^0(E)=0$:} Suppose that $\HH^0(E)\neq0$. Then $\HH^0(E)=\mathcal{O}_S(1)$ with $S\hookrightarrow \mathbb{P}^3$ a subvariety due to $\HH^0(E)$ being a quotient of $\mathcal{O}(1)$. Its twisted ideal $J(1)$ is the image of the map $f:\mathcal{O}^{\oplus 4} \rightarrow \mathcal{O}(1)$ defining $\tau_{>-2}(E)$ and we can see that its space of global sections has dimension $h^0(J(1))\leq 3$, making it so that $h^0(\ker(f))\neq 0$.

Moreover, this implies that $h^0(\HH^{-1}(E))\neq0$ because it is a quotient of $\ker(f)$ by $\mathcal{O}(-1)$, and therefore there exists a non zero map $\mathcal{O}[1]\rightarrow \HH^{-1}(E)[1] \rightarrow E$. Hence, there exists a destabilizing object for $E$ in $\bar{R}_-$ by Lemma \ref{Lemma-Determinant-Calculation}, an absurd.  

\end{example}

\section{Appendix: Computational Observations}\label{Sec-Comp}

As seen in the paper, the construction of examples of Ext-exceptional collections is very important to establishing quiver regions. One way of constructing these collections is by shifting the objects in a strong exceptional collection by the appropriate degree. For $\mathbb{P}^2$, it is a known result by Rudakov \cite{Rud} that the only strong exceptional collections we can create are mutations of the canonical exceptional collection , proved to be strong by Beilinson \cite{Beil}. 

\begin{question}
Is there an Ext-exceptional collection in $\mathbb{P}^3$ satisfying the upper-half plane for any point in $\mathbb{H}$ with $\alpha>2$ and $s=\frac{1}{3}$?
\end{question}

A positive answer to this question would guarantee that the moduli space of Bridgeland semistable objects are projective, for every geometric stability condition, for example.

In this appendix we describe a computational approach that lead to Question $1$. There is no known list of all full exceptional collections in $\mathbb{P}^3$ but there exists a conjecture about the transitivity of the action of the group $B_n \rtimes \mathbb{Z}^n$ over the space of strong exceptional collections, where $B_n$ is the $n$-braid group acting by mutations and $\mathbb{Z}^n$ is the action by shifts, see \cite{BP}. Our idea was to use mutations to inductively define strong exceptional collections, with easy to compute Chern characters due to their additive property, and test for the upper-half plane condition using a parametrization of lines going through the origin.

We start by defining the notion of mutation of an exceptional collection and their computational properties. 

\begin{definition}
Let $\{E_0,E_1\}$ be a exceptional pair in $\Db$, a left (resp. right) mutation of the exceptional pair is the object $\mathcal{L}_{E_0}E_1$ (resp. $\mathcal{R}_{E_1}E_0$) defined by the following distinguished triangles \[ \mathcal{L}_{E_0}E_1 \rightarrow \Hom^\bullet(E_0,E_1)\otimes E_0 \rightarrow E_1\] \[(E_0 \rightarrow \Hom^\bullet(E_0,E_1)^\ast \otimes E_1  \rightarrow \mathcal{R}_{E_1}E_0).\]
\end{definition}

\begin{definition}
For an exceptional collection $\mathcal{E}=\{E_0,...,E_k\}$ in $\Db$ we can define the Left and Right mutations of $\mathcal{E}$ by \[ \mathcal{L}_i\mathcal{E}=\{E_0,...,E_{i-1},\mathcal{L}_{E_i}E_{i+1},E_i,E_{i+2},...,E_k\} \] \[\mathcal{R}_i\mathcal{E}=\{E_0,...,E_{i+1},\mathcal{R}_{E_{i+1}}E_i,E_{i+2},...,E_k\} \]
\end{definition}

It is known that a mutation of an exceptional collection is also an exceptional collection, they also preserve fullness. That is not true for strong exceptional collections, in general, but in our case this is actually true.

The following theorem describe what happens when mutating a strong exceptional collection and the nature of the objects in a strong exceptional collection.

\begin{theorem}\label{Teo-Bondal}\cite[Section 9]{Bon1}
Suppose that $\mathcal{E}=\langle E_0,...,E_n\rangle$ is a full exceptional collection of sheaves in $\Dbal(X)$ for a $n$-dimensional manifold $X$. Then $\mathcal{E}$ is a  strong exceptional collection and any mutation of $\mathcal{E}$ is a full strong exceptional collection.
\end{theorem}

\begin{theorem}\label{Teo-Pos}\cite{Pos}
Let $X$ be a smooth projective variety for which $n=\dim \Knum(X)-1=\dim(X)$. Then for any full strong exceptional collection $E_0,...,E_n$ in $\Dbal(X)$, the objects $E_i$ are shifts of locally free sheaves by the same number $a \in \mathbb{Z}$. 
\end{theorem}

\begin{rmk}\label{Mu-Inequality}
For a smooth $n$ dimensional projective variety X, the Theorems \ref{Teo-Bondal} and \ref{Teo-Pos} can be applied in conjunction so that $\Hom^\bullet(E_i,E_{i+1})=\Hom^0(E_i,E_{i+1})\neq 0$ if $\mathcal{E}=\{E_i\}_i$ is a strong exceptional collection, because otherwise $\mathcal{L}_{E_i}E_{i+1}=E_{i+1}[1]$ and that would be a contradiction due to every strong exceptional collection, in this case $\mathcal{L}_i(\mathcal{E})$, consisting of sheaves shifted by the same number in $\mathbb{Z}$. Another consequence of this is that if the $E_i$ are all $\mu$-stable then $\mu(E_i)<\mu(E_{i+1})$.
\end{rmk}

With this at hand, we can describe an algorithm to produce candidates of full Ext-exceptional collections capable of generating a heart of a bounded $t$-structure satisfying the upper-half plane condition. As a consequence, we were not able to find a point $(\beta,\alpha)$ for $s=1/3$ in the upper-half plane of stability conditions $\mathbb{H}$ with $\alpha>2$ where a single example of full Ext-exceptional collection was capable of satisfying the conditions imposed by the algorithm(a numerical version of the upper-half plane condition).

\begin{algorithm}
We start with a set pre-determined of Chern characters of a given strong exceptional collection of locally free sheaves $\mathcal{E}=\{E_0,...,E_3\}$ that is the strings \[c[i]=(\ch_0(E_i),\ch_1(E_i),\ch_2(E_i),\ch_3(E_i))=(\ch_j(E_i))_{0\leq j \leq n}.\] With this string at hand we can calculate the Chern character of the mutations $\mathcal{L}_i\mathcal{E}$ and $\mathcal{R}_j\mathcal{E}$ by realizing that in a strong exceptional collection of locally free sheaves we have $\Hom^\bullet(E_i,E_j)=\Hom(E_i,E_j)$ and \[\dim(\Hom(E_i,E_j)) =\chi(E_i,E_j)=\chi(E_i\otimes E_j^\vee, \opn),\] where $\chi(E_i\otimes E_j^\vee, \opn)$ can be calculated by the Hirzebruch-Riemann-Roch, as in \cite[Theorem 9.3]{BM}, and $\ch(E_i \otimes E_j^\vee)=\ch(E_i)\cdot \ch(E_j)^\vee$ with $(\ch(E_j)^\vee)_i=(-1)^i\ch_i(E_j)$.

Once we know $\Tilde{c}[j]=(\ch_j(\mathcal{L}_k\mathcal{E}))_{0\leq j \leq 3}$ or $\Bar{c}[j]=(\ch_j(\mathcal{R}_k\mathcal{E}))_{0\leq j \leq 3}$, for all $k=0,1,2$, we can calculate if this string satisfy the numerical conditions to which any object $Q_j$ with $\ch(Q_j)=\Tilde{c}[j]$ or $\ch(Q_j)=\Bar{c}[j]$ and $Q_j \in <\mathcal{A},\mathcal{A}[1]>$ is subject to, for all $j$, such as
\begin{itemize}
    \item[(i)] $\nu_{\beta,\alpha}(Q_0)>0$ and $\mu(Q_0)>\beta$ (if $Q_0[3] \in \mathcal{A}[1]$)
    \item[(ii)] Either: \subitem $\nu_{\beta,\alpha}(Q_1)>0$ and $\mu(Q_1)>\beta$ (if $Q_1[2] \in \mathcal{A}$), \subitem $\nu_{\beta,\alpha}(Q_1)\leq0$ and $\mu(Q_1)>\beta$ (if $Q_1[2] \in \mathcal{A}[1]$) or \subitem $\nu_{\beta,\alpha}(Q_1)>0$ and $\mu(Q_1)\leq \beta$ (if $Q_1[2] \in \mathcal{A}[1]$)
    \item[(iii)] Either: \subitem$\nu_{\beta,\alpha}(Q_2)<0$ and $\mu(Q_2)>\beta$ (if $Q_2[1] \in \mathcal{A}$), \subitem $\nu_{\beta,\alpha}(Q_2)\leq0$ and $\mu(Q_2)\leq \beta$ (if $Q_2[1] \in \mathcal{A}$) or \subitem  $\nu_{\beta,\alpha}(Q_2)>0$ and $\mu(Q_2)\leq\beta$ (if $Q_2[1] \in \mathcal{A}[1]$)
    \item[(iv)] $\nu_{\beta,\alpha}(Q_3)>0$ and $\mu(Q_3)>\beta$ (if $Q_3 \in \mathcal{A}$)
\end{itemize}

If any of the conditions $(i),(ii),(iii),(iv)$ is not satisfid the algorithm returns $0$ with respect to this mutation $\mathcal{L}_k \mathcal{E}$ or $\mathcal{R}_k\mathcal{E}$, otherwise we continue the algorithm to see if $\mathcal{L}_i\mathcal{E}$ or $\mathcal{R}_i\mathcal{E}$ satisfy the upper-half plane condition. This is done by testing if all the images of $Z_{\beta,\alpha,1/3}(\Tilde{c}[j])$ or $Z_{\beta,\alpha,1/3}(\Bar{c}[j])$ are above a given line with slope $\phi$, with $\phi$ varying from $0$ to $\pi$ with increments of $0.01\pi$. In the end, if the collection mutated satisfies $(i),(ii),(iii),(iv)$ and the upper-half plane condition then the algorithm returns $1$. This can be iterated further for any of the mutated collections.

\end{algorithm}

One other computational approach to the problem of finding actual $\lambda$-walls in a quiver region $R$ for a fixed dimension vector $[a,b,c,d]$ is to determine its numerical walls, these will be a finite set due to Lemma \ref{Lemma-Dimension}, and apply the numerical conditions

\begin{itemize}
    \item[(i)](Positivity) $\rho_{\beta,\alpha}(v)\geq 0$,
    \item[(ii)](Generalized Bogomolov inequality) $Q_{\beta,\alpha}(v)\geq 0$
\end{itemize}

to any possible destabilizing dimension vector $v=[a',b',c',d']$. These necessary conditions restrict the numerical walls to a finer set which can be studied using Proposition \ref{lem}.

\begin{definition}
A region C of the upper-half plane $\mathbb{H}$ is said to be $\lambda$-bounded with respect to $E \in \Dbal(X)$ if there exists $a,b \in \mathbb{R}$ such that, for every $(\beta,\alpha) \in \mathbb{H}$ and a fixed $s \in \mathbb{R}_{>0}$, the $\mathbb{C}$-slope of $E$ is in $[a,b] \subset (0,1)$.
\end{definition}

\begin{example}
For any object $E \in \Db$, one clear example of a $C$ is a compact subset of the upper-half plane not intersecting $\Theta_E$, 
\end{example}

\begin{proposition}\label{lem}
Let $E \in \Db$ and $P,Q$ be points in a $\lambda$-bounded region $C$ such that exists $\gamma:[0,1] \rightarrow C$ a continuous curve with endpoints in $P$ and $Q$ respectively. If $E \notin \mathcal{A}^P$ and $E \in \mathcal{P}(\lambda_{Q,s}(E)) \subset \mathcal{A}^Q$, then there exists $w \in (0,1]$ such that for every $t \in [w,1]$, $E \in \mathcal{A}^{\gamma(t)}$ and $E$ is $\lambda_{\gamma(w),s}$-unstable. 
\end{proposition}

\begin{proof}
We start by noting that $\phi_{\beta,\alpha}^+(E)$ and $\phi_{\beta,\alpha}^-(E)$ are continuous functions on $\mathbb{H}$ and therefore $f(t):[0,1] \rightarrow \mathbb{R}$
\begin{center}
    $f(t)= \phi_{\gamma(t)}^+(E)-\phi_{\gamma(t)}^-(E)$
\end{center}
is a non-negative continuous function. By the condition in $P$ we know that $f$ is not a trivial function with image $\{0\}$ and since it is a continuous image of a connected set, we know that $f([0,1])=[0,c]$ for some $c \in \mathbb{R}_{>0}$. 

Now let $\epsilon>0$ be a real number which makes $[a-\epsilon,b+\epsilon] \subset(0,1)$. It is clear that for every $t \in f^{-1}([0,\epsilon])$ we have $E \in \mathcal{A}^{\gamma(t)}$, we just have to choose the connected component $W$ of $f^{-1}([0,\epsilon])$ containing $1$ and this set will be $W=[w,1]$ for some $w \in R_{>0}$.

\end{proof}

The consequence of Proposition \ref{lem} is that for an object to be in a actual $\lambda$-wall, that is for it to became semistable, it has to first be unstable in $\mathcal{A}^{\gamma(w)}$. This allows for an inductive approach in this quiver regions, where we use the fact that twisting by $\mathcal{O}(-1)$ the quiver region $R_1$ intersect the original region $R_1$ at the line $\{-1\} \times (0,0.7)$, and if we can prove that an object is semistable at region $R_1$ then we can draw curves through the twisted region that don't go inside the regions "numerically possible" to prove that none of this numerical $\lambda$-walls are real.

\end{document}